\documentclass{amsart}

\usepackage{epsfig}
\usepackage{amsmath}
\usepackage{amssymb}
\usepackage{amscd}
\usepackage{graphicx}
\usepackage{color}

\newtheorem{thm}{Theorem}[section]
\newtheorem{lem}[thm]{Lemma}
\newtheorem{prop}[thm]{Proposition}

\def \N {\mathbb N}

\def \Z {\mathbb Z}
\def \R {\mathbb R}
\def \E {\mathbb E}

\numberwithin{equation}{section}

\begin{document}

\title{local variational principle concerning entropy\\ of a sofic group action}

\date{September 12, 2011}

\author{Guohua Zhang}

\address{School of Mathematical Sciences and LMNS, Fudan University, Shanghai 200433, China}

\email{chiaths.zhang@gmail.com}

\begin{abstract}
Recently Lewis Bowen introduced a notion of entropy for measure-preserving actions of countable sofic groups admitting a generating measurable partition with finite entropy; and then David Kerr and Hanfeng Li developed an operator-algebraic approach to actions of countable sofic groups not only on a standard probability space but also on a compact metric space, and established the global variational principle concerning measure-theoretic and topological entropy in this sofic context. By localizing these two kinds of entropy, in this paper we prove a local version of the global variational principle for any finite open cover of the space, and show that these local measure-theoretic and topological entropy coincide with their classical counterparts when the acting group is an infinite amenable group.
\end{abstract}

\maketitle

\markboth{local variational principle concerning entropy of a sofic group action}{Guohua Zhang}


\section{Introduction}

Dynamical system theory is the study of the qualitative properties of group actions on spaces with certain structures.
In order to distinguish between two measure-preserving $\Z$-actions which are spectrally isomorphic, in 1958 Kolmogorov introduced an isomorphism invariant which is called measure-theoretic entropy in ergodic theory \cite{Ko58}. And then the concept of topological entropy was introduced in 1965 by Adler, Konheim and McAndrew for topological $\Z$-actions \cite{AKM}.
From then on the relationship between these two kinds of entropy has gained a lot of attention.

Let $(X, T)$ be a topological $\Z$-action, that is, $T: X\rightarrow X$ is a homeomorphism. Denote by $h_\xi (T, X)$ and $h (T, X)$ the $\xi$-measure-theoretic entropy and topological entropy of $X$, respectively, where $\xi$ is a $T$-invariant Borel probability measure over $X$. In his 1969 paper \cite{Goodwyn} Goodwyn showed that $h_\mu (T, X)\le h (T, X)$ for any $T$-invariant measure $\mu$ over $X$, and later Goodman \cite{Goodman} proved $\sup h_\nu (T, X)\ge h (T, X)$, where the supremum is taken over all $T$-invariant measures, completing the classical variational principle for $(X, T)$. See \cite{Misiu} for a short proof.

For topological $\Z$-actions, starting with the study of the topological analogue of Kolmogorov systems \cite{B1},
the entropy concept can be localized by defining entropy pairs and tuples (even entropy sets and entropy points) both in topological and in measure-theoretical situations \cite{B2, BHMMR, DYZ, HY, YZ}. To study the relation of entropy pairs and tuples in both settings, a local variational inequality \cite{BGH}, a local variational relation \cite{HY} and finally local variational principles concerning entropy \cite{GW4, R} which refine the classical variational principle were found. It was then generalized to the relative setting of a given factor map between topological $\Z$-actions \cite{HYZ1, HYZ2}, actions of a countable discrete amenable group on a compact metric space \cite{HYZ} and continuous bundle random dynamical systems of an infinite countable discrete amenable group action \cite{DZRDS}, respectively. For the whole  theory of it see the recent survey \cite{GY} by Glasner and Ye and the references therein.

Recently Lewis Bowen introduced a notion of entropy for measure-preserving actions of a countable discrete sofic group admitting a generating measurable partition with finite entropy \cite{BowenLJAMS10}. The basic idea is to model the dynamics of a measurable partition of the probability space by partitions of a finite space on which the group acts in an approximate way according to the definition of soficity. Given a fixed sequence of sofic approximations, the cardinality of the set of all such model partitions is then used to asymptotically generate a number, which is shown to be invariant over all generating measurable partitions with finite entropy. It may depend though on the choice of sofic approximation sequences, yielding in general a collection of entropy invariants. However, in the case that the acting group is amenable and the action admits a generating measurable partition with finite entropy, the sofic measure-theoretic entropy coincides with the classical Kolmogorov-Sinai entropy for all choices of a sofic approximation sequence \cite{BowenLETDS11}.

Just after that, in the spirit of Bowen's measure-theoretic entropy, David Kerr and Hanfeng Li developed an operator-algebraic approach to sofic entropy, which applies to not only sofic measure-theoretic entropy but also actions of any countable sofic group on a compact metric space, and established the global variational principle concerning measure-theoretic and topological entropy in this sofic context \cite{KLInvention1x}. In fact, they extended Bowen's sofic measure-theoretic entropy to actions that need not necessarily admit a generating measurable partition with finite entropy. The key to doing all of this is to view the corresponding dynamics at the operator level and replace Bowen's combinatorics of measurable partitions with an analysis of multiplicative or approximately multiplicative linear maps that are approximately equivariant. When the acting group is amenable, by expressing both measure-theoretic and topological entropy in terms of the dynamics on the space itself, these global invariants are shown to coincide with their classical counterparts, independently of the choice of a sofic approximation sequence \cite{KLAJM}.

To study the local properties of entropy for actions of a countable discrete sofic group on a compact metric space, the following question arises naturally: does the local variational principle concerning entropy hold in the sofic setting?

The main purpose of present paper is to answer this question by localizing the main results obtained in \cite{KLInvention1x, KLAJM}.
 Observe that, here, we are not to consider the separated and spanning subsets as in \cite{KLAJM}, alternatively, we are willing to consider finite open covers of the state space. The idea is not new, and it has been used first to consider sofic mean dimension by Li in \cite{Lsoficdim}.
Given an action of a countable discrete sofic group on a compact metric space and a sofic approximation sequence, in this paper for a finite open cover we introduce the measure-theoretic and topological entropy, and then we prove the local variational principle concerning these two kinds of entropy for any finite open cover of the space. In the case that the acting group is infinite and amenable, these local invariants are proved to coincide with their classical counterparts, independently of the choice of a sofic approximation sequence.
 In the proof of this equivalence, following the ideas of \cite[Section 6]{KLAJM} we use the ergodic decomposition of local measure-theoretic entropy \cite[Lemma 3.12]{HYZ} (for the case of $G= \Z$ see for example \cite[Lemma 4.8]{HY} or \cite[Theorem 8.4]{W}, see also \cite[Theorem 5.3]{HYZ1}), and so here we require $G$ to be infinite (see Lemma \ref{1107131641} and Lemma \ref{1107131627} for details).
When the acting group is finite, as the global measure-theoretic entropy is a conjugacy invariant \cite{KLInvention1x}, these two kinds of global measure-theoretical invariants were proved to be equivalent \cite[Lemma 6.5 and Lemma 6.6]{KLAJM}; whereas, the problem if these two kinds of local measure-theoretical invariants are equivalent remains open.
 Similar to the global case, the basis for our analysis of the equivalence of these local invariants is a sofic approximation version of the Rokhlin lemma of Ornstein and Weiss proved in \cite[Section 4]{KLAJM}, see also \cite[Section 4]{Elek} a form treating more generally the finite graphs.
As we could obtain the global invariants by taking the supremum for local invariants over all finite open covers of the space, the global results, including the global variational principle, follow directly from the local ones.

The paper is organized as follows. In section 2 we recall the amenability and soficity of group and introduce the equivalent definitions of measure-theoretic and topological entropy in the sofic setting, and then in section 3 we give some basic properties of them. In section 4 we prove the local variational principle concerning these two kinds of entropy. As a direct application of the obtained local variational principle, in section 5 we introduce and discuss entropy tuples both in topological and in measure-theoretical situations for actions of a countable discrete sofic group on a compact metric space. In section 6 we are willing to compare these local invariants with their classical counterparts in the setting of the group being amenable, in particular, we prove that if the group is infinite and amenable then they coincide with the classical ones.

\section{Preliminaries} \label{preli}

In this section, after recalling the amenability and soficity of group, for actions of a countable discrete sofic group on a compact metric space we introduce definitions of measure-theoretic and topological entropy for any finite open cover of the space in the sofic setting, which are equivalent to those from \cite{KLAJM}.

For each $a\in \R$ denote by $[a]$ and $\lceil a\rceil$ the largest integer smaller than $a$ and the smallest integer lager than $a$, respectively.

For $d\in \N$ we write $Sym (d)$ for the group of permutations of $\{1, \cdots, d\}$.

For a set $Z$ denote by $\mathcal{F}_Z$ the set of all non-empty finite subsets of $Z$ and by $|Z|$ its cardinality. Let $d\in \N$, we put $Z^d= \{(x_1, \cdots, x_d): x_1\in Z, \cdots, x_d\in Z\}$ and $\Delta_d (Z)= \{(x_1, \cdots, x_d): x_1= \cdots= x_d\in Z\}$.
For any map $\sigma: Z\rightarrow Sym (d), x\mapsto \sigma_x$ with some $d\in \N$, let $Y\subseteq Z, B\subseteq \{1, \cdots, d\}$ and $c\in \{1, \cdots, d\}$, we write $\sigma (Y) B= \{\sigma_y (b): y\in Y, b\in B\}$ and $\sigma (Y) c= \{\sigma_y (c): y\in Y\}$.

Recall that a countable discrete group $G$ is \emph{amenable}, if there exists a sequence $\{F_n: n\in \N\}\subseteq \mathcal{F}_G$ such that, for all $g_1, g_2\in G$,
$$\lim_{n\rightarrow \infty} \frac{|g_1 F_n g_2\Delta F_n|}{|F_n|}= 0.$$
Such a sequence $\{F_n: n\in \N\}\subseteq \mathcal{F}_G$ is called a \emph{F\o lner sequence} for $G$.
See for example \cite{Folner} and \cite[I.\S 0 and I.\S 1]{OW}.
Each finite discrete group is amenable; and if $\{F_n: n\in \N\}\subseteq \mathcal{F}_G$ is a F\o lner sequence for a finite group $G$ then $F_n= G$ for all large enough $n\in \N$, while, if $\{F_n: n\in \N\}\subseteq \mathcal{F}_G$ is a F\o lner sequence for an infinite amenable group $G$ then $\lim\limits_{n\rightarrow \infty} |F_n|= \infty$.

Throughout the whole paper, we will fix $G$ to be a countable discrete sofic group (with unit $e$). That is, \cite{KLAJM} there are a sequence $\{d_i: i\in \N\}\subseteq \N$ and a sequence $\{\sigma_i: i\in \N\}$ of maps $\sigma_i: G\rightarrow Sym (d_i), g\mapsto \sigma_{i, g}$ which is \emph{asymptotically multiplicative} and \emph{free} in the sense that
$$\lim_{i\rightarrow \infty} \frac{1}{d_i} |\{a\in \{1, \cdots, d_i\}: \sigma_{i, s t} (a)= \sigma_{i, s} \sigma_{i, t} (a)\}|= 1$$
for all $s, t\in G$ and
$$\lim_{i\rightarrow \infty} \frac{1}{d_i} |\{a\in \{1, \cdots, d_i\}: \sigma_{i, s} (a)\neq \sigma_{i, t} (a)\}|= 1$$
for all distinct $s, t\in G$. Such a sequence $\{\sigma_i: i\in \N\}$ with $\lim\limits_{i\rightarrow \infty} d_i= \infty$ is referred to as a \emph{sofic approximation sequence} of $G$, and we will fix it throughout the paper. Observe that the condition $\lim\limits_{i\rightarrow \infty} d_i= \infty$ is essential for the global variational principle concerning entropy of a sofic group action \cite{KLInvention1x} and is automatic if $G$ is infinite.
  Note that if $G$ is amenable then it is sofic, as one can easily construct a sofic approximation sequence from a F\o lner sequence.

Sofic groups were defined implicitly by Gromov in \cite{Gromov99} and explicitly by Weiss in \cite{Weiss-sofic}. Recently Pestov has written a beautiful up-to-date survey \cite{Pestov08} on sofic groups and their siblings, hyperlinear groups.

From now on, we assume that $G$ acts continuously on a compact metric space $(X, \rho)$ as a group of self-homeomorphisms over $X$.
Denote by $\mathcal{M} (X)$ the set of all Borel probability measures over $X$, by $\mathcal{M} (X, G)$ the set of all $G$-invariant Borel probability measures over $X$ and by $\mathcal{M}^e (X, G)$ the set of all ergodic $G$-invariant Borel probability measures over $X$, respectively. All of them are equipped with the well-known weak star topology. It is well known that if $G$ is amenable then $\mathcal{M}^e (X, G)\neq \emptyset$ always holds; whereas, in general it may happen $\mathcal{M} (X, G)= \emptyset$. For example for the rank two free group $F_2$, there exists a compact metric space $Y$ such that, $F_2$ acts as a group of homeomorphisms on $Y$, while, each Borel probability measure over $Y$ is not $F_2$-invariant.

Denote by $\mathcal{B}_X$ the Borel $\sigma$-algebra of $X$,
and by $C (X)$ the set of all real-valued continuous functions over $X$ which is equipped with the supremum norm $||\cdot||$.

By a \emph{cover} of $X$ we mean a family of subsets of $X$ with the whole space as its union. If elements of a cover are pairwise disjoint, then it is called a \emph{partition}.
Denote by $\mathcal{C}_X^o$ the set of all finite open covers of $X$,
 by $\mathcal{C}_X$ the set of all finite Borel covers of $X$, and by $\mathcal{P}_X$ the set of all finite Borel partitions of $X$.
For $\mathcal{V}_i\in \mathcal{C}_X, i= 1, 2$, we say that $\mathcal{V}_1$ is \emph{finer} than $\mathcal{V}_2$ if each element of $\mathcal{V}_1$ is contained in some element of $\mathcal{V}_2$ (denoted by $\mathcal{V}_1\succeq \mathcal{V}_2$).
For $\mathcal{V}\in \mathcal{C}_X$ and $\emptyset\neq K\subseteq X$ we set $N (\mathcal{V}, K)$ to be the minimal cardinality of sub-families of $\mathcal{V}$ covering $K$, we also set $N (\mathcal{V}, \emptyset)= 0$ by convention.

Now let's recall the following equivalent definitions from \cite{KLAJM}.

For $F\in \mathcal{F}_G, \delta> 0$ and a map $\sigma: G\rightarrow Sym (d), g\mapsto \sigma_g$ with $d\in \N$, we set
$$X^d_{F, \delta, \sigma}= \left\{(x_1, \cdots, x_d)\in X^d: \max_{s\in F} \sqrt{\sum_{i= 1}^d \frac{1}{d} \rho^2 (s x_i, x_{\sigma_s (i)})}< \delta\right\}.$$
Now let $L\in \mathcal{F}_{C (X)}$ and $\mu\in \mathcal{M} (X)$, we set
$$X^d_{F, \delta, \sigma, \mu, L}= \left\{(x_1, \cdots, x_d)\in X^d_{F, \delta, \sigma}: \max_{f\in L} |\frac{1}{d} \sum_{i= 1}^d f (x_i)- \mu (f)|< \delta\right\}.$$

Let $\mathcal{U}\in \mathcal{C}_X^o$. Put (with convention $\log 0 = - \infty$)
$$h_{F, \delta} (G, \mathcal{U})= \limsup_{i\rightarrow \infty} \frac{1}{d_i} \log N \left(\mathcal{U}^{d_i}, X^{d_i}_{F, \delta, \sigma_i}\right),$$
$$h_{F, \delta, \mu, L} (G, \mathcal{U})= \limsup_{i\rightarrow \infty} \frac{1}{d_i} \log N \left(\mathcal{U}^{d_i}, X^{d_i}_{F, \delta, \sigma_i, \mu, L}\right)\le h_{F, \delta} (G, \mathcal{U}),$$
and set
$$h (G, \mathcal{U})= \inf_{F\in \mathcal{F}_G} \inf_{\delta> 0} h_{F, \delta} (G, \mathcal{U})\le \log N (\mathcal{U}, X),$$
$$h_\mu (G, \mathcal{U})= \inf_{L\in \mathcal{F}_{C (X)}} \inf_{F\in \mathcal{F}_G} \inf_{\delta> 0} h_{F, \delta, \mu, L} (G, \mathcal{U})\le h (G, \mathcal{U}).$$
We define $h (G, \mathcal{U})$ (may take the value of $- \infty$) to be the \emph{topological entropy of $\mathcal{U}$} for the system $(X, G)$. Then we define the \emph{topological entropy} of $(X, G)$ as
$$h (G, X)= \sup_{\mathcal{U}\in \mathcal{C}_X^o} h (G, \mathcal{U}).$$
We define $h_\mu (G, \mathcal{U})$ (may also take the value of $- \infty$) to be the \emph{$\mu$-measure-theoretic entropy of $\mathcal{U}$} for the system $(X, G)$. Then we define the \emph{$\mu$-measure-theoretic entropy} of $(X, G)$ as
$$h_\mu (G, X)= \sup_{\mathcal{U}\in \mathcal{C}_X^o} h_\mu (G, \mathcal{U}).$$

Now let $\mathcal{V}\in \mathcal{C}_X$ (i.e. elements of $\mathcal{V}$ need not to be open sets). We define the \emph{$\mu$-measure-theoretic entropy of $\mathcal{V}$} for the system $(X, G)$ as
$$h_\mu (G, \mathcal{V})= \sup_{\mathcal{U}\in \mathcal{C}_X^o, \mathcal{V}\succeq \mathcal{U}} h_\mu (G, \mathcal{U}).$$
Observe that we could always find some $\mathcal{U}\in \mathcal{C}_X^o$ with $\mathcal{V}\succeq \mathcal{U}$ and, for $\mathcal{U}_i\in \mathcal{C}_X^o, i= 1, 2$ with $\mathcal{U}_1\succeq \mathcal{U}_2$ by the previous definition $h_\mu (G, \mathcal{U}_1)\ge h_\mu (G, \mathcal{U}_2)$. The above definition is well defined, and for $\mathcal{V}_i\in \mathcal{C}_X, i= 1, 2$ with $\mathcal{V}_1\succeq \mathcal{V}_2$ one has $h_\mu (G, \mathcal{V}_1)\ge h_\mu (G, \mathcal{V}_2)$.
Moreover, one has
$$h_\mu (G, X)= \sup_{\mathcal{V}\in \mathcal{C}_X} h_\mu (G, \mathcal{V})= \sup_{\alpha\in \mathcal{P}_X} h_\mu (G, \alpha).$$
Remark that the global measure-theoretic entropy is a conjugacy invariant \cite{KLInvention1x}, as two dynamically generating sequences
have the same entropy \cite[Theorem 4.5]{KLInvention1x}.

\section{Basic properties of sofic entropy} \label{1109111126}

In this section we are to discuss some basic properties of sofic entropy.

First let's check that these invariants are independent of the selection of a compatible metric over $X$. In fact, this follows from the following observation.

\begin{prop} \label{1107101636}
Let $\rho_1$ and $\rho_2$ be two compatible metrics over $X$ and $F\in \mathcal{F}_G$. Then for each $\delta_2> 0$ there exists $\delta_2\ge \delta_1> 0$ such that
$$X^d_{F, \delta_1, \sigma; \rho_1}\subseteq X^d_{F, \delta_2, \sigma; \rho_2}$$
      for each map $\sigma: G\mapsto Sym (d)$ with some $d\in \N$. Here, we use $X^d_{F, \delta_1, \sigma; \rho_1}$ and $X^d_{F, \delta_2, \sigma; \rho_2}$ to emphasize the corresponding metrics $\rho_1$ and $\rho_2$, respectively.
\end{prop}

We could obtain directly Proposition \ref{1107101636} from the proof of \cite[Lemma 2.4]{Lsoficdim}, and so here we omit its proof.

It was proved implicitly in the proof of \cite[Theorem 6.1]{KLInvention1x} that if $\mu\in \mathcal{M} (X)$ is not $G$-invariant then $h_\mu (G, X)= - \infty$, equivalently, if $h_\mu (G, X)\ge 0$ then $\mu\in \mathcal{M} (X, G)$, see also the proof of Theorem \ref{1107091735} along the same idea.

In the remainder of this section, we are to give some easy estimate of Bowen's measure-theoretic sofic entropy for a finite measurable cover.

Before proceeding, we need the following combinatorial result.

\begin{lem} \label{1107182209}
Let $p_1\ge \cdots\ge p_n> 0$ satisfy $\sum\limits_{k= 1}^n p_k= 1$ and $\epsilon> 0$. Then there exists $\eta> 0$ small enough such that once $\Lambda$ is a finite set with $|\Lambda|$ large enough, then
$$\log |\Gamma_{\eta, p_1, \cdots, p_n}|\le |\Lambda| \left(- \sum_{k= 1}^n p_k \log p_k+ 2 \epsilon\right),$$
where $\Gamma_{\eta, p_1, \cdots, p_n}$ denotes the set of all partitions $\{\gamma_1, \cdots, \gamma_n, \gamma_{n+ 1}\}$ of $\Lambda$ with
$$\max_{k= 1}^n |\frac{|\gamma_k|}{|\Lambda|}- p_k|< \eta.$$
\end{lem}
\begin{proof}
Set $p= - \sum\limits_{k= 1}^n p_k \log p_k$. By Stirling's approximation formula there is $\eta_1> 0$ with $\eta_1< p_n$ such that once $(a_1, \cdots, a_n)\in \N^n$ with $\sum\limits_{k= 1}^n a_k$ large enough satisfies
$$\max_{k= 1}^n |\frac{a_k}{a_1+ \cdots+ a_n} - p_k|\le \eta_1$$
then
$$\binom{a_1+ \cdots+ a_n}{a_n}\binom{a_1+ \cdots+ a_{n- 1}}{a_{n- 1}} \cdots \binom{a_1+ a_2}{a_2}\le e^{(a_1+ \cdots+ a_n) (p+ \epsilon)}.$$

Let $\Lambda$ be a finite set. Observe $\sum\limits_{k= 1}^n |\gamma_k|\ge |\Lambda| (1- n \eta)$ for any $\{\gamma_1, \cdots, \gamma_n, \gamma_{n+ 1}\}$ $\in \Gamma_{\eta, p_1, \cdots, p_n}$.
There is $\eta> 0$ with $n \eta< < 1$ such that once $|\Lambda|$ is large enough then
 $$\max_{k= 1}^n |\frac{|\gamma_k|}{|\gamma_1|+ \cdots+ |\gamma_n|}- p_k|< \eta_1$$
for any $\{\gamma_1, \cdots, \gamma_n, \gamma_{n+ 1}\}\in \Gamma_{\eta, p_1, \cdots, p_n}$ and
\begin{equation} \label{1107182345}
|\Lambda|^n \sum_{Q= \lceil |\Lambda| (1- n \eta)\rceil}^{|\Lambda|} \binom{|\Lambda|}{Q}\le e^{|\Lambda| \epsilon}.
\end{equation}
Here, applying again the Stirling's approximation formula such an $\eta> 0$ satisfying \eqref{1107182345} exists. Denote by $\Lambda_n$ the set of all $(a_1, \cdots, a_n)\in \N^n$ such that $a_1= |\gamma_1|, \cdots, a_n= |\gamma_n|$ for some $\{\gamma_1, \cdots, \gamma_n, \gamma_{n+ 1}\}\in \Gamma_{\eta, p_1, \cdots, p_n}$. From the above discussions, we have that once $|\Lambda|$ is large enough then
\begin{eqnarray*} \label{0007091735}
|\Gamma_{\eta, p_1, \cdots, p_n}|&\le & \sum_{Q= \lceil |\Lambda| (1- n \eta)\rceil}^{|\Lambda|} |\left\{\{\gamma_1, \cdots, \gamma_n, \gamma_{n+ 1}\}\in \Gamma_{\eta, p_1, \cdots, p_n}: \sum_{k= 1}^n |\gamma_k|= Q\right\}|\nonumber \\
&= & \sum_{Q= \lceil |\Lambda| (1- n \eta)\rceil}^{|\Lambda|} \binom{|\Lambda|}{Q} \sum_{(a_1, \cdots, a_n)\in \Lambda_n, \sum\limits_{k= 1}^n a_k= Q} \binom{Q}{a_n}\cdots \binom{a_2+ a_1}{a_2}\nonumber \\
&\le & \sum_{Q= \lceil |\Lambda| (1- n \eta)\rceil}^{|\Lambda|} \binom{|\Lambda|}{Q} Q^n e^{Q (p+ \epsilon)}\ (\text{by the selection of $\eta_1, \eta$})\nonumber \\
&\le & \sum_{Q= \lceil |\Lambda| (1- n \eta)\rceil}^{|\Lambda|} \binom{|\Lambda|}{Q} |\Lambda|^n e^{|\Lambda| (p+ \epsilon)}.
\end{eqnarray*}
Combined with \eqref{1107182345}, we obtain the conclusion readily.
\end{proof}

Let $\alpha\in \mathcal{P}_X$ and $\mu\in \mathcal{M} (X)$. Set (by convention $0 \log 0= 0$)
\begin{equation*}
H_\mu (\alpha)= - \sum_{A\in \alpha} \mu (A) \log \mu (A).
\end{equation*}

Then we have the following estimation.

\begin{lem} \label{1100131627}
Let $\mathcal{U}\in \mathcal{C}_X^o, \mu\in \mathcal{M} (X, G), \epsilon> 0$ and $\alpha\in \mathcal{P}_X$ satisfy $\alpha\succeq \mathcal{U}$. Then there exist $\delta> 0$ and $L\in \mathcal{F}_{C (X)}$ such that
$h_{\{e\}, \delta, \mu, L} (G, \mathcal{U})\le H_\mu (\alpha)+ 3 \epsilon$.
\end{lem}
\begin{proof}
Say $A_1, \cdots, A_n, n\in \N$ to be the set of all atoms of $\alpha$ with positive $\mu$-measure. As $\alpha\succeq \mathcal{U}$, for each $k= 1, \cdots, n$ there exists $U_k\in \mathcal{U}$ with $A_k\subseteq U_k$.
Set $\tau= \min\limits_{k= 1}^n \mu (A_k)> 0$ and let $\tau> \eta> 0$ be given by Lemma \ref{1107182209} for $\mu (A_1), \cdots, \mu (A_n)$.

Let $\kappa> 0$ such that $\kappa\le \eta$ and $|\mathcal{U}|^\kappa< e^\epsilon$. Let $\delta> 0$ such that $n \delta< \frac{\kappa}{2}$.

By the regularity of $\mu$, for each $k= 1, \cdots, n$ there exist a function $0\le f_k\le 1$ in $C (X)$, a closed subset $B_k\subseteq A_k$ and an open subset $C_k\supseteq B_k$ with $\overline{C_k}\subseteq U_k$ such that $f_k|_{B_k}= 1, f_k|_{C_k^c}= 0, \mu (B_k)\ge \mu (A_k)- \delta$ and $\overline{C_1}, \cdots, \overline{C_n}$ are pairwise disjoint.
Set $L= \{f_1, \cdots, f_n\}\in \mathcal{F}_{C (X)}$.

Now
let $\sigma: G\rightarrow Sym (d)$ be a good enough sofic approximation for $G$ with some $d\in \N$, and so $d$ is large enough. For $(x_1, \cdots, x_d)\in X^d_{\{e\}, \delta, \sigma, \mu, L}$, we consider $\Lambda_k= \{a\in \{1, \cdots, d\}: x_a\in C_k\}$ for each $k= 1, \cdots, n$. Then, for each $k= 1, \cdots, n$,
\begin{equation} \label{1107242136}
\frac{|\Lambda_k|}{d}\ge \frac{1}{d} \sum_{i= 1}^d f_k (x_i)\ge \mu (f_k)- \delta\ge \mu (B_k)- \delta\ge \mu (A_k)- 2 \delta,
\end{equation}
and hence $|\Lambda_{n+ 1}|\le 2 n \delta d$, where $\Lambda_{n+ 1}= \{1, \cdots, d\}\setminus \bigcup\limits_{j= 1}^n \Lambda_j$, and
\begin{equation} \label{1107242139}
\frac{|\Lambda_k|}{d}\le \mu (A_k)+ 2 n \delta\ (\text{applying \eqref{1107242136} to each $k'\in \{1, \cdots, d\}\setminus \{k\}$}),
\end{equation}
as $\Lambda_1, \cdots, \Lambda_n$ are pairwise disjoint. Observe $2 n \delta< \kappa\le \eta$ and $d$ is large enough, by the selection of $\eta$ and $\kappa$ and using \eqref{1107242136} and \eqref{1107242139} we have
\begin{equation*} \label{1107242203}
\log N (\mathcal{U}^d, X^d_{\{e\}, \delta, \sigma, \mu, L})\le d (H_\mu (\alpha)+ 2 \epsilon)+ 2 n \delta d \log |\mathcal{U}|\le d (H_\mu (\alpha)+ 3 \epsilon).
\end{equation*}
Then the conclusion follows from the above estimation.
\end{proof}

For $\mathcal{V}_1, \mathcal{V}_2\in \mathcal{C}_X$, set $\mathcal{V}_1\vee \mathcal{V}_2= \{V_1\cap V_2: V_1\in \mathcal{V}_1, V_2\in \mathcal{V}_2\}$. It works similarly for any given finite elements from $\mathcal{C}_X$.

Let $\mathcal{V}\in \mathcal{C}_X$ and $F\in \mathcal{F}_G$, we write $\mathcal{V}_F= \bigvee\limits_{g\in F} g^{- 1} \mathcal{V}$. Let $\nu\in \mathcal{M} (X)$. Set
$$H_\nu (\mathcal{V})= \inf_{\alpha\in \mathcal{P}_X, \alpha\succeq \mathcal{V}} H_\nu (\alpha).$$
In fact, by \cite[Proposition 6]{R}, there exists a finite family $P (\mathcal{V})\subseteq \{\alpha\in \mathcal{P}_X: \alpha\succeq \mathcal{V}\}$ (depends only on $\mathcal{V}$, independent of $\nu\in \mathcal{M} (X)$) such that
\begin{equation} \label{1107160141}
H_\nu (\mathcal{V})= \min_{\alpha\in P (\mathcal{V})} H_\nu (\alpha).
\end{equation}
For $\mathcal{V}_1, \mathcal{V}_2\in \mathcal{C}_X$, obviously $H_\nu (\mathcal{V}_1)\ge H_\nu (\mathcal{V}_2)$ once $\mathcal{V}_1\succeq \mathcal{V}_2$.

Thus, as a direct corollary of Lemma \ref{1100131627}, we have:

\begin{prop} \label{1107242213}
Let $\mathcal{V}\in \mathcal{C}_X$ and $\mu\in \mathcal{M} (X, G)$. Then $h_\mu (G, \mathcal{V})\le H_\mu (\mathcal{V})$.
\end{prop}

Remark that as a direct corollary of \cite[Lemma 5.1]{BowenLJAMS10}, for $\mu\in \mathcal{M} (X, G)$ if $(X, G, \mu)$ admits some generating partition $\alpha\in \mathcal{P}_X$ in the sense that for each $A\in \mathcal{B}_X$ there exists $B\in \mathcal{B}_{X, \alpha}$ satisfying $\mu (A\Delta B)= 0$, where $\mathcal{B}_{X, \alpha}$ denotes the smallest $G$-invariant sub-$\sigma$-algebra of $\mathcal{B}_X$ containing all atoms of $\alpha$, then $h_\mu (G, X)\le H_\mu (\alpha)$. The author thanks Li for pointing out this point.

\section{Local variational principle concerning sofic entropy}

The global variational principle concerning entropy of a sofic group action is proved by Kerr and Li \cite[Theorem 6.1]{KLInvention1x}. In this section, we aim to prove a local version of it following the line of \cite{KLInvention1x}.

Our local variational principle concerning sofic entropy is stated as follows.

\begin{thm} \label{1107091735}
Let $\mathcal{U}\in \mathcal{C}_X^o$. Then
$$h (G, \mathcal{U})= \max_{\mu\in \mathcal{M} (X, G)} h_\mu (G, \mathcal{U}).$$
In right hand of the above formula, we set it as $- \infty$ by convention if $\mathcal{M} (X, G)= \emptyset$.
\end{thm}

 Observe that the global variational principle follows from the local one by taking the supremum on both hands over all finite open covers of the state space. We should remark that the proof of \cite[Theorem 6.1]{KLInvention1x} gives another local variational principle different from the result we are to prove.

Before proving Theorem \ref{1107091735}, let's first prove the following result.

\begin{lem} \label{1107091749}
Let $\mathcal{U}\in \mathcal{C}_X^o$ and $F\in \mathcal{F}_G, L\in \mathcal{F}_{C (X)}, \delta> 0$. Then there exists
 $\nu\in \mathcal{M} (X)$ satisfying $h_{F, \delta, \nu, L} (G, \mathcal{U})\ge h_{F, \delta} (G, \mathcal{U})$.
\end{lem}
\begin{proof}
Observe that $\mathcal{M} (X)$ is a compact metric space (induced naturally by the metric $\rho$ on $X$), there exists $D\in \mathcal{F}_{\mathcal{M} (X)}$ such that once $d\in \N$ and $(x_1, \cdots, x_d)\in X^d$ there exists $\nu (x_1, \cdots, x_d)\in D$ with
$$\max_{f\in L} |\frac{1}{d} \sum_{i= 1}^d f (x_i)- \int_X f d \nu (x_1, \cdots, x_d)|< \delta.$$

Now for any map $\sigma: G\rightarrow Sym (d)$ with some $d\in \N$, we introduce
$$X^d_{F, \delta, \sigma, L} (\xi)= \{(x_1, \cdots, x_d)\in X^d_{F, \delta, \sigma}: \nu (x_1, \cdots, x_d)= \xi\}\subseteq X^d_{F, \delta, \sigma, \xi, L}$$
for each $\xi\in D$.
Observe
$$\bigcup_{\xi\in D} X^d_{F, \delta, \sigma, L} (\xi)= X^d_{F, \delta, \sigma},$$
there exists $\nu_\sigma\in D$ such that
$$N (\mathcal{U}^d, X^d_{F, \delta, \sigma, \nu_\sigma, L})\ge \frac{N (\mathcal{U}^d, X^d_{F, \delta, \sigma})}{|D|}.$$

Now apply the above discussion to each $\sigma_i, i\in \N$, and by taking a sub-sequence we may assume that $\nu= \nu_{\sigma_i}$ for each $i\in \N$. Then $h_{F, \delta, \nu, L} (G, \mathcal{U})\ge h_{F, \delta} (G, \mathcal{U})$ follows directly from the definitions.
\end{proof}

Now let's turn to the proof of Theorem \ref{1107091735}.

\begin{proof}[Proof of Theorem \ref{1107091735}]
By the convention, it is direct to obtain
$$h (G, \mathcal{U})\ge \sup_{\mu\in \mathcal{M} (X, G)} h_\mu (G, \mathcal{U})$$
from the definitions. Thus it suffices to prove
$$\max_{\mu\in \mathcal{M} (X, G)} h_\mu (G, \mathcal{U})\ge h (G, \mathcal{U}).$$
By our convention we may assume that $h (G, \mathcal{U})> -\infty$, and so we only need to find some $\mu\in \mathcal{M} (X, G)$ with $h_\mu (G, \mathcal{U})\ge h (G, \mathcal{U})$.

In $\mathcal{F}_G$ we take a sequence $F_1\subseteq F_2\subseteq \cdots$ with
$\bigcup\limits_{i= 1}^\infty F_i= G$
and let $\{f_n: n\in \N\}$ be a countable dense subset in $C (X)$. Now set $L_n= \{f_1, \cdots, f_n\}$ for each $n\in \N$.

Let $n\in \N$. By Lemma \ref{1107091749}, there exists $\nu_n\in \mathcal{M} (X)$ such that
\begin{equation} \label{1107092037}
h_{F_n, \frac{1}{n}, \nu_n, L_n} (G, \mathcal{U})\ge h_{F_n, \frac{1}{n}} (G, \mathcal{U})\ge h (G, \mathcal{U}).
\end{equation}
As $\mathcal{M} (X)$ is a compact metric space, by taking a sub-sequence we may assume that $\{\nu_n: n\in \N\}$ converges to $\mu$ in $\mathcal{M} (X)$. Now we are to prove that the constructed $\mu$ has the required property.

Let $F\in \mathcal{F}_G, L\in \mathcal{F}_{C (X)}$ and $\delta> 0$. By the above constructions there exists $n\in \N$ such that
\begin{enumerate}

\item $\frac{3}{n}< \delta$ and $F\subseteq F_n$;

\item $\nu_n$ is sufficiently close to $\mu$; and

\item for each $f\in L$ there exists $g_n\in L_n$ such that $g_n$ is sufficiently close to $f$.
\end{enumerate}
In particular, $X^d_{F_n, \frac{1}{n}, \sigma, \nu_n, L_n}\subseteq X^d_{F, \delta, \sigma, \mu, L}$ for all maps $\sigma: G\rightarrow Sym (d)$ for some $d\in \N$. Now applying it to each $\sigma_i, i\in \N$ we obtain
$$h_{F, \delta, \mu, L} (G, \mathcal{U})\ge h_{F_n, \frac{1}{n}, \nu_n, L_n} (G, \mathcal{U})\ge h (G, \mathcal{U})\ (\text{using \eqref{1107092037}}).$$
By the arbitrariness of $F, \delta, L$ we obtain $h_\mu (G, \mathcal{U})\ge h (G, \mathcal{U})$.

Just as remarked in section \ref{1109111126}, the constructed $\mu$ should be $G$-invariant as $h (G, \mathcal{U})> - \infty$ by the assumption. Whereas, following the ideas of \cite[Theorem 6.1]{KLInvention1x} here we are to give a detailed direct proof of it for completeness.

In order to prove $\mu\in \mathcal{M} (X, G)$. Let $p\in C (X)$ and $g\in G$. We only need prove $\mu (p)= \mu (p\circ g)$.
Let $\epsilon> 0$. As $p\in C (X)$ there exists $\epsilon\ge \delta_1> 0$ such that
\begin{equation} \label{1107092058}
\sup_{(y_1, y_2)\in X^2, \rho (y_1, y_2)< \delta_1} |p (y_1)- p (y_2)|< \epsilon.
\end{equation}
Now let $k\in \N$ such that $\frac{2 ||p||}{k^2}< \epsilon$.
As
$$h_{\{g\}, \frac{\delta_1}{k}, \mu, \{p, p\circ g\}} (G, \mathcal{U})\ge h_\mu (G, \mathcal{U})\ge h (G, \mathcal{U})> - \infty,$$
there exists at least a map $\sigma: G\rightarrow Sym (d)$ for some $d\in \N$ such that
$$X^d_{\{g\}, \frac{\delta_1}{k}, \sigma, \mu, \{p, p\circ g\}}\neq \emptyset.$$
 Let $(x_1, \cdots, x_d)\in X^d_{\{g\}, \frac{\delta_1}{k}, \sigma, \mu, \{p, p\circ g\}}$. Then
\begin{equation} \label{1107092105}
\sqrt{\sum_{i= 1}^d \frac{1}{d} \rho^2 (g x_i, x_{\sigma_g (i)})}< \frac{\delta_1}{k}
\end{equation}
and
\begin{equation} \label{1107092106}
\max_{f\in \{p, p\circ g\}} |\frac{1}{d} \sum_{i= 1}^d f (x_i)- \mu (f)|< \delta_1.
\end{equation}
Consider
$J= \{i\in \{1, \cdots, d\}: \rho (g x_i, x_{\sigma_g (i)})\ge \delta_1\}$.
By \eqref{1107092105}, one has
\begin{equation} \label{1107101610}
\frac{\delta_1}{k}> \delta_1\cdot \sqrt{\frac{|J|}{d}},\ \text{and so}\ |J|< \frac{d}{k^2}.
\end{equation}
As if $i\in \{1, \cdots, d\}\setminus J$ then $\rho (g x_i, x_{\sigma_g (i)})< \delta_1$ and so $|p (g x_i)- p (x_{\sigma_g (i)})|< \epsilon$ (using \eqref{1107092058}).
Thus, one has
\begin{eqnarray} \label{1107092109}
|\frac{1}{d} \sum_{i= 1}^d p (x_i)- \frac{1}{d} \sum_{i= 1}^d p\circ g (x_i)|&= & |\frac{1}{d} \sum_{i= 1}^d p (x_{\sigma_g (i)})- \frac{1}{d} \sum_{i= 1}^d p (g x_i)|\nonumber \\
&\le & \frac{1}{d} \sum_{i= 1}^d |p (x_{\sigma_g (i)})- p (g x_i)|\nonumber \\
&\le & \frac{1}{d} (|\{1, \cdots, d\}\setminus J| \epsilon+ |J| 2 ||p||)\nonumber \\
&< & \epsilon+ \frac{2 ||p||}{k^2}\ (\text{using \eqref{1107101610}})< 2 \epsilon.
\end{eqnarray}
Then combining \eqref{1107092106} and \eqref{1107092109}, we obtain
$$|\mu (p)- \mu (p\circ g)|< 2 \delta_1+ 2 \epsilon\le 4 \epsilon.$$
By the arbitrariness of $\epsilon$ we obtain $\mu (p)= \mu (p\circ g)$. This finishes the proof.
\end{proof}

\section{Entropy tuples of a sofic group action}

As a direct corollary of the local variational principle proved in the previous section, in this section we discuss some local properties of entropy for actions of a countable discrete sofic group on a compact metric space.

Let $(x_1, \cdots, x_n)\in X^n\setminus \Delta_n (X), n\in \N\setminus \{1\}$ and $\mu\in \mathcal{M} (X)$.
\begin{enumerate}

\item $(x_1, \cdots, x_n)$ is called an \emph{entropy $n$-tuple} of $(X, G)$ if $h (G, \mathcal{U})> 0$ once $\mathcal{U}= \{U_1^c, \cdots, U_n^c\}\in \mathcal{C}_X^o$ where $U_i$ is a closed neighborhood of $x_i, i= 1, \cdots, n$.

\item $(x_1, \cdots, x_n)$ is called a \emph{$\mu$-entropy $n$-tuple} of $(X, G)$ if $h_\mu (G, \mathcal{U})> 0$ once $\mathcal{U}= \{U_1^c, \cdots, U_n^c\}\in \mathcal{C}_X^o$ where $U_i$ is a closed neighborhood of $x_i, i= 1, \cdots, n$.
\end{enumerate}
Denote by $E_n (X, G)$ and $E_n^\mu (X, G)$ the set of all entropy $n$-tuples and all $\mu$-entropy $n$-tuples of $(X, G)$, respectively.

As if $\mu\in \mathcal{M} (X)$ is not $G$-invariant then $h_\mu (G, X)= - \infty$ and so $E_n^\mu (X, G)= \emptyset$ by the definitions. In the following we are only interested in the case of $\mu\in \mathcal{M} (X, G)$.

It is not hard to obtain the following results along the lines of \cite{B2, BHMMR}.

\begin{prop} \label{1107092127}
Let $\mu\in \mathcal{M} (X, G)$ and $\mathcal{U}= \{U_1, \cdots, U_n\}\in \mathcal{C}_X^o, \mathcal{V}= \{V_1, \cdots, V_n\}$ $\in \mathcal{C}_X, n\in \N\setminus \{1\}$.
\begin{enumerate}

\item If $h (G, \mathcal{U})> 0$ then there exists $x_i\in U_i^c$ for each $i= 1, \cdots, n$ such that $(x_1, \cdots, x_n)\in E_n (X, G)$.

\item If $h_\mu (G, \mathcal{V})> 0$ then there exists $x_i\in V_i^c$ for each $i= 1, \cdots, n$ such that $(x_1, \cdots, x_n)\in E_n^\mu (X, G)$.
\end{enumerate}
\end{prop}

\begin{prop} \label{1107092128}
Let $\mu\in \mathcal{M} (X, G)$ and $n\in \N\setminus \{1\}$. Then
\begin{enumerate}

\item $E_n^\mu (X, G)\subseteq E_n (X, G)$.

\item Both $E_n (X, G)\cup \Delta_n (X)$ and $E_n^\mu (X, G)\cup \Delta_n (X)$ are closed subsets of $X^n$.

\item $h (G, X)> 0$ if and only if $E_n (X, G)\neq \emptyset$; $h_\mu (G, X)> 0$ if and only if $E_n^\mu (X, G)\neq \emptyset$.
\end{enumerate}
\end{prop}

Moreover, with the help of Theorem \ref{1107091735}, we obtain:

\begin{thm} \label{1107092143}
Let $n\in \N\setminus \{1\}$. Then
$$E_n (X, G)= \overline{\bigcup_{\mu\in \mathcal{M} (X, G)} E_n^\mu (X, G)}\setminus \Delta_n (X).$$
\end{thm}
\begin{proof}
By Proposition \ref{1107092128}, we have readily
$$E_n (X, G)\supseteq \overline{\bigcup_{\mu\in \mathcal{M} (X, G)} E_n^\mu (X, G)}\setminus \Delta_n (X).$$
Now we are to obtain the conclusion by proving
\begin{equation} \label{five-star}
\overline{\bigcup_{\mu\in \mathcal{M} (X, G)} E_n^\mu (X, G)}\setminus \Delta_n (X)\supseteq E_n (X, G).
\end{equation}

Let $(x_1, \cdots, x_n)\in E_n (X, G)$. Once $m\in \N$ is large enough, we may find a closed neighborhood $U_{i, m}$ of $x_i$ with diameter at most $\frac{1}{m}$ for each $i= 1, \cdots, n$ such that $\mathcal{U}_m\doteq \{U_{1, m}^c, \cdots, U_{n, m}^c\}\in \mathcal{C}_X^o$, which implies $h (G, \mathcal{U}_m)> 0$, and so by Theorem \ref{1107091735} there exists $\mu_m\in \mathcal{M} (X, G)$ with $h_{\mu_m} (G, \mathcal{U}_m)> 0$, hence
$$(U_{1, m}\times \cdots\times U_{n, m})\cap E_n^{\mu_m} (X, G)\neq \emptyset\ \text{(using Proposition \ref{1107092127})}.$$
 It is easy to obtain \eqref{five-star} from the above discussions, which finishes the proof.
\end{proof}

\section{Comparing them to the usual ones for amenable group actions}

In this section we are to compare those introduced sofic entropy for a finite cover with their classical counterparts in the setting of the group being amenable.
Thus, throughout this section, additionally we assume that the countable discrete sofic group $G$ is amenable. We prove that if the group is infinite and amenable then they coincide with the classical ones.
Whereas, different from the global case \cite{KLAJM}, when the acting group is finite, the problem if these two kinds of local measure-theoretical invariants are equivalent remains open.

Let $\mathcal{U}\in \mathcal{C}_X^o$. Recall that $\mathcal{U}_F, F\in \mathcal{F}_G$ is introduced in section \ref{1109111126} as $\bigvee\limits_{g\in F} g^{- 1} \mathcal{U}$. As guaranteed by the well-known Ornstein-Weiss lemma \cite[Theorem 6.1]{Lin-Wei}, the usual topological entropy of $\mathcal{U}$ for $(X, G)$ when considering an amenable group action, denoted by $h^a (G, \mathcal{U})$, is the limit of
$$(0\le)\ \frac{1}{|F|} \log N (\mathcal{U}_F, X)\ (\le \log N (\mathcal{U}, X)\le |\mathcal{U}|)$$
as $F\in \mathcal{F}_G$ becomes more and more left invariant in the sense that for each $\epsilon> 0$ there exist $K\in \mathcal{F}_G$ and $\delta> 0$ such that
$$|h^a (G, \mathcal{U})- \frac{1}{|F|} \log N (\mathcal{U}_F, X)|< \epsilon$$
once $F\in \mathcal{F}_G$ satisfies $|K F\Delta F|\le \delta |F|$.
Similarly, let $\mathcal{V}\in \mathcal{C}_X$ and $\mu\in \mathcal{M} (X, G)$. Denote by $h^a_\mu (G, \mathcal{V})$
the usual $\mu$-measure-theoretic entropy of $\mathcal{V}$ for $(X, G)$ when considering an amenable group action. That is, $h_\mu^a (G, \mathcal{V})$ is the limit of
$$(0\le)\ \frac{1}{|F|} H_\mu (\mathcal{V}_F)\ (\le \log |\mathcal{V}|)$$
as $F\in \mathcal{F}_G$ becomes more and more left invariant. See \cite{HYZ, MO, OW, WZ} for details.

In this section, we are to prove the following results.

\begin{thm} \label{1107101454}
Let $\mathcal{U}\in \mathcal{C}_X^o$. Then $h (G, \mathcal{U})= h^a (G, \mathcal{U})$.
\end{thm}

\begin{thm} \label{1107101633}
Let $\mathcal{V}\in \mathcal{C}_X$ and $\mu\in \mathcal{M} (X, G)$. Assume that $G$ is infinite. Then $h_\mu (G, \mathcal{V})= h_\mu^a (G, \mathcal{V})$.
\end{thm}

Let $Y$ be a finite set, $\{A_i: i\in I\}\subseteq \{\emptyset\}\cup \mathcal{F}_Y$ and $\delta\ge 0$. $\{A_i: i\in I\}$ is said to \emph{$\delta$-cover} or be a \emph{$\delta$-covering} of $Y$ if
$|\bigcup\limits_{i\in I} A_i|\ge \delta |Y|$.
$\{A_i: i\in I\}$ are \emph{$\epsilon$-disjoint} if there exist pairwise disjoint subsets $B_i\subseteq A_i$ with $|B_i|\ge (1- \epsilon) |A_i|$ for each $i\in I$.

It holds the Rokhlin Lemma for sofic approximation sequences \cite[Lemma 4.5]{KLAJM}.

\begin{lem} \label{1107110016}
Let $\Gamma$ be a countable group and $0\le \tau< 1, 0< \eta< 1$. Then there are an $l\in \N$ and $\eta', \eta''> 0$ such that, whenever $e\in E_1\subseteq \cdots\subseteq E_l$ are finite subsets of $\Gamma$ with $|E_{k- 1}^{- 1} E_k\setminus E_k|\le \eta' |E_k|$ for $k= 2, \cdots, l$, there exists $e\in E\in \mathcal{F}_\Gamma$ such that for every good enough sofic approximation $\sigma: \Gamma\rightarrow Sym (d)$ for $\Gamma$ with some $d\in \N$ (i.e. $\sigma: \Gamma\rightarrow Sym (d)$ is a map with $B\subseteq \{1, \cdots, d\}$ satisfying $|B|\ge (1- \eta'') d$ and
$$\sigma_{s t} (a)= \sigma_s \sigma_t (a), \sigma_s (a)\neq \sigma_{s'} (a), \sigma_e (a)= a$$
for all $a\in B$ and $s, t, s'\in E$ with $s\neq s'$), and any set $V\subseteq \{1, \cdots, d\}$ with $|V|\ge (1- \tau) d$, there exist $C_1, \cdots, C_l\subseteq V$ such that
\begin{enumerate}

\item the sets $\sigma (E_k) C_k, k\in \{1, \cdots, l\}$ are pairwise disjoint;

\item $\{\sigma (E_k) C_k: k\in \{1, \cdots, l\}\}$ $(1- \tau- \eta)$-covers $\{1, \cdots, d\}$;

\item $\{\sigma (E_k) c: c\in C_k\}$ is $\eta$-disjoint for each $k\in \{1, \cdots, l\}$; and

\item for every $k\in \{1, \cdots, l\}$ and $c\in C_k$, $E_k\ni s\mapsto \sigma_s (c)$ is bijective.
\end{enumerate}
\end{lem}

Before proceeding, we also need the following easy observation.

\begin{lem} \label{1107122331}
Let $F\in \mathcal{F}_G$ and $\mathcal{U}\in \mathcal{C}_X^o$. Then there exists $\delta> 0$ such that
$$X_{F, \delta}= \left\{(x_s)_{s\in F}\in X^F: \max_{s\in F} \rho (x_s, s x)< \delta\ \text{for some}\ x\in X\right\}$$
can be covered by at most $N (\mathcal{U}_F, X)$ elements of $\mathcal{U}^F$.
\end{lem}
\begin{proof}
Obviously, there exists $\mathcal{V}\subseteq \mathcal{U}^F$ such that $|\mathcal{V}|\le N (\mathcal{U}_F, X)$ and
\begin{equation*} \label{1107122343}
\cup \mathcal{V}\supseteq X_F\ \text{where}\ X_F= \{(s x)_{s\in F}: x\in X\}.
\end{equation*}
For example, let $\mathcal{W}\subseteq \mathcal{U}_F$ such that $|\mathcal{W}|= N (\mathcal{U}_F, X)$ and $\cup \mathcal{W}= X$. Now for each $W\in \mathcal{W}$, as $W\in \mathcal{U}_F$, say $W= \bigcap\limits_{s\in F} s^{- 1} U (s)$ with $U (s)\in \mathcal{U}$ for each $s\in F$, we set $\widehat{W}= \prod\limits_{s\in F} U (s)\in \mathcal{U}^F$. Then we can take $\mathcal{V}$ to be $\{\widehat{W}: W\in \mathcal{W}\}$.

Note that $\cup \mathcal{V}$ is an open subset of $X^F$ and $X_F\subseteq X^F$ is a non-empty closed subset, there exists $\delta> 0$ such that $X_{F, \delta}\subseteq \cup \mathcal{V}$. This finishes the proof.
\end{proof}

Then, following the ideas of \cite[Lemma 5.1]{KLAJM} we have:

\begin{lem} \label{1107120120}
Let $\mathcal{U}\in \mathcal{C}_X^o$. Then $h (G, \mathcal{U})\le h^a (G, \mathcal{U})$.
\end{lem}
\begin{proof}
Let $\epsilon> 0$. Then there exist $K\in \mathcal{F}_G$ and $\delta'> 0$ such that
$$\frac{1}{|F|} \log N (\mathcal{U}_F, X)\le h^a (G, \mathcal{U})+ \epsilon$$
once $F\in \mathcal{F}_G$ satisfies $|K F\Delta F|\le \delta' |F|$.

We choose $1> \eta> 0$ small enough such that
\begin{equation} \label{1107130122}
\frac{h^a (G, \mathcal{U})+ \epsilon}{1- \eta}+ 2 \eta \log |\mathcal{U}|\le h^a (G, \mathcal{U})+ 2 \epsilon.
\end{equation}
Now let $l\in \N$ and $\eta'> 0$ be as given by Lemma \ref{1107110016} with respect to $\tau= \eta$ and $\eta$. In $\mathcal{F}_G$ we take $e\in F_1\subseteq \cdots\subseteq F_l$ such that $|F_{k- 1}^{- 1} F_k\setminus F_k|\le \eta' |F_k|$ for $k= 2, \cdots, l$ and $|K F_k\Delta F_k|\le \delta' |F_k|$ for $k= 1, \cdots, l$. As the group $G$ is amenable, such subsets $F_1, \cdots, F_l$ must exist. Thus
\begin{equation} \label{1107122326}
\max_{k= 1}^l \frac{1}{|F_k|} \log N (\mathcal{U}_{F_k}, X)\le h^a (G, \mathcal{U})+ \epsilon.
\end{equation}
For each $k= 1, \cdots, l$ let $\delta_k> 0$ be as given by Lemma \ref{1107122331} with respect to $F_k$ and $\mathcal{U}$.
Take $\delta> 0$ such that $\delta\le \min\{\delta_1^2, \cdots, \delta_l^2, \frac{\eta}{|F_l|}\}$ and if $d\in \N$ is large enough then
\begin{equation} \label{1107130020}
\sum_{j= 0}^{[|F_l| \delta d]} \binom{d}{j}< (1+ \epsilon)^d.
\end{equation}

Now let $\sigma: G\rightarrow Sym (d)$ be a good enough sofic approximation for $G$ with some $d\in \N$. If $(x_1, \cdots, x_d)\in X^d_{F_l, \delta, \sigma}$ then
$$\max_{s\in F_l} \sqrt{\sum_{i= 1}^d \frac{1}{d} \rho^2 (s x_i, x_{\sigma_s (i)})}< \delta,$$
which implies that $|J (x_1, \cdots, x_d, F_l)|\ge (1- |F_l| \delta) d$, where
$$J (x_1, \cdots, x_d, F_l)= \left\{i\in \{1, \cdots, d\}: \max_{s\in F_l} \rho (s x_i, x_{\sigma_s (i)})< \sqrt{\delta}\right\}.$$
Now denote by $\Theta$ the set of all subsets of $\{1, \cdots, d\}$ with at least $(1- |F_l| \delta) d$ many elements and for each $\theta\in \Theta$ by $X^d_{F_l, \delta, \sigma, \theta}$ the set of all $(x_1, \cdots, x_d)\in X^d_{F_l, \delta, \sigma}$ with $J (x_1, \cdots, x_d, F_l)= \theta$. Then
\begin{equation} \label{1107130014}
|\Theta|= \sum_{j= 0}^{[|F_l| \delta d]} \binom{d}{j}< (1+ \epsilon)^d\ (\text{using \eqref{1107130020}}),
\end{equation}
as $\sigma$ is good enough and so $d\in \N$ is large enough.

Let $\theta\in \Theta$. As $\sigma$ is good enough, by Lemma \ref{1107110016} there exist $C_1, \cdots, C_l\subseteq \theta$ with
\begin{enumerate}

\item the sets $\sigma (F_k) C_k, k\in \{1, \cdots, l\}$ are pairwise disjoint;

\item $\{\sigma (F_k) c: c\in C_k\}$ is $\eta$-disjoint for each $k= 1, \cdots, l$;

\item $\{\sigma (F_k) C_k: k\in \{1, \cdots, l\}\}$ $(1- 2 \eta)$-covers $\{1, \cdots, d\}$; and

\item for every $k\in \{1, \cdots, l\}$ and $c\in C_k$, $F_k\ni s\mapsto \sigma_s (c)$ is bijective.
\end{enumerate}
Set $J_\theta= \{1, \cdots, d\}\setminus \cup \{\sigma (F_k) C_k: k\in \{1, \cdots, l\}\}$. Then
\begin{equation} \label{1107130110}
|J_\theta|\le 2 \eta d\ \text{and}\ \sum_{k= 1}^l |F_k|\cdot |C_k|\le \frac{1}{1- \eta} \sum_{k= 1}^l |\sigma (F_k) C_k|\le \frac{d}{1- \eta}.
\end{equation}
Now let $k= 1, \cdots, l$. For any $c_k\in C_k$, as $C_k\subseteq \theta$ and $F_k\subseteq F_l$, by the selection of $\delta$ it is direct to see that we can cover
\begin{eqnarray*}
& & \{(x_i)_{i\in \sigma (F_k) c_k}: (x_1, \cdots, x_d)\in X^d_{F_l, \delta, \sigma, \theta}\} \\
&\subseteq & \left\{(x_i)_{i\in \sigma (F_k) c_k}: \max_{s\in F_k} \rho (x_{\sigma_s (c_k)}, s x)< \delta_k\ \text{for some}\ x\in X\right\}
\end{eqnarray*}
by at most $N (\mathcal{U}_{F_k}, X)$ elements of $\mathcal{U}^{\sigma (F_k) c_k}$, and so it is not hard to cover
$$\{(x_i)_{i\in \sigma (F_k) C_k}: (x_1, \cdots, x_d)\in X^d_{F_l, \delta, \sigma, \theta}\}$$
using at most $N (\mathcal{U}_{F_k}, X)^{|C_k|}$ elements of $\mathcal{U}^{\sigma (F_k) C_k}$. Thus
\begin{eqnarray} \label{1107131006}
\log N (\mathcal{U}^d, X^d_{F_l, \delta, \sigma, \theta})&\le & \sum_{k= 1}^l |C_k| \log N (\mathcal{U}_{F_k}, X)+ |J_\theta| \log |\mathcal{U}|\nonumber \\
&\le & (h^a (G, \mathcal{U})+ \epsilon) \sum_{k= 1}^l |C_k|\cdot |F_k|+ |J_\theta| \log |\mathcal{U}|\ (\text{using \eqref{1107122326}})\nonumber \\
&\le & d \left(\frac{h^a (G, \mathcal{U})+ \epsilon}{1- \eta}+ 2 \eta \log |\mathcal{U}|\right)\ (\text{using \eqref{1107130110}}) \nonumber \\
&\le & d (h^a (G, \mathcal{U})+ 2 \epsilon)\ (\text{using \eqref{1107130122}}).
\end{eqnarray}

Combining \eqref{1107130014} with \eqref{1107131006} we obtain
$$\log N (\mathcal{U}^d, X^d_{F_l, \delta, \sigma})\le d (h^a (G, \mathcal{U})+ 2 \epsilon+ \log (1+ \epsilon)).$$
By the arbitrariness of $\epsilon$ we obtain the conclusion.
\end{proof}

We also have \cite[Lemma 4.6]{KLAJM}, which is an improved version of Lemma \ref{1107110016} for an amenable group. Recall that the group $G$ is amenable.

\begin{lem} \label{1107112242}
Let $0\le \tau< 1$ and $0< \eta< 1$. Then there are an $l\in \N$ and $F_1, \cdots, F_l\in \mathcal{F}_G$ which are sufficiently two-sided invariant such that for every good enough sofic approximation $\sigma: G\rightarrow Sym (d)$ for $G$ with some $d\in \N$ and any set $V\subseteq \{1, \cdots, d\}$ with $|V|\ge (1- \tau) d$, there exist $C_1, \cdots, C_l\subseteq V$ such that
\begin{enumerate}

\item the sets $\sigma (F_k) C_k, k\in \{1, \cdots, l\}$ are pairwise disjoint;

\item $\{\sigma (F_k) C_k: k\in \{1, \cdots, l\}\}$ $(1- \tau- \eta)$-covers $\{1, \cdots, d\}$; and

\item for every $k\in \{1, \cdots, l\}$, the map $F_k\times C_k\ni (s, c)\mapsto \sigma_s (c)$ is bijective.
\end{enumerate}
\end{lem}

Thus, following the ideas of \cite[Lemma 5.2]{KLAJM} we have:

\begin{lem} \label{1107111610}
Let $\mathcal{U}\in \mathcal{C}_X^o$. Then $h (G, \mathcal{U})\ge h^a (G, \mathcal{U})$.
\end{lem}
\begin{proof}
Let $\theta> 0$ and $F\in \mathcal{F}_G, \delta> 0$. Now we are to finish the proof by proving
\begin{equation} \label{1107111631}
\frac{1}{d} \log N (\mathcal{U}^d, X^d_{F, \delta, \sigma})\ge h^a (G, \mathcal{U})- 2 \theta
\end{equation}
 once $\sigma: G\rightarrow Sym (d)$ is a good enough sofic approximation for $G$ with some $d\in \N$.

 Let $M> 0$ large enough and $\delta'> 0$ small enough such that the diameter of the space $X$ is at most $M$ and
\begin{equation} \label{1107111645}
\sqrt{\delta'} M< \frac{\delta}{2}\ \text{and}\ (1- \delta') h^a (G, \mathcal{U})\ge h^a (G, \mathcal{U})- \theta.
\end{equation}

Applying Lemma \ref{1107112242}, there are an $l\in \N$ and $F_1, \cdots, F_l\in \mathcal{F}_G$, which are sufficiently left invariant so that
\begin{equation} \label{1107120045}
\min_{k= 1}^l \min_{s\in F} \frac{|s^{- 1} F_k\cap F_k|}{|F_k|}\ge 1- \delta'
\end{equation}
and
\begin{equation} \label{1107112321}
\min_{k= 1}^l \frac{1}{|F_k|} \log N (\mathcal{U}_{F_k}, X)\ge h^a (G, \mathcal{U})- \theta,
\end{equation}
such that once $\sigma: G\rightarrow Sym (d)$ is a good enough sofic approximation for $G$ with some $d\in \N$ then there exist $C_1, \cdots, C_l\subseteq \{1, \cdots, d\}$ satisfying
\begin{enumerate}

\item the sets $\sigma (F_k) C_k, k\in \{1, \cdots, l\}$ are pairwise disjoint;

\item $\{\sigma (F_k) C_k: k\in \{1, \cdots, l\}\}$ $(1- \delta')$-covers $\{1, \cdots, d\}$;

\item for every $k\in \{1, \cdots, l\}$, the map $F_k\times C_k\ni (s, c)\mapsto \sigma_s (c)$ is bijective; and

\item for all $k\in \{1, \cdots, l\}$ and $s\in F, s_k\in F_k, c_k\in C_k$, $\sigma_{s s_k} (c_k)= \sigma_s \sigma_{s_k} (c_k)$.
\end{enumerate}
Remark again that the group $G$ is amenable, such subsets $F_1, \cdots, F_l$ must exist.

Now assume that $\sigma: G\rightarrow Sym (d)$ is a good enough sofic approximation for $G$ with some $d\in \N$ and let $C_1, \cdots, C_l\subseteq \{1, \cdots, d\}$ be constructed as above.
Let $(y_1, \cdots, y_l)$ be any $l$-tuple with $y_k\in X^{C_k}, k\in \{1, \cdots, l\}$. From the construction of $C_1, \cdots, C_l$, it is not hard to see that there exists at least one point $(x_1, \cdots, x_d)\in X^d$ such that once $i\in \sigma (F_k) C_k$ for some $k\in \{1, \cdots, l\}$, say $i= \sigma_{s_k} (c_k)$ with $s_k\in F_k$ and $c_k\in C_k$, then $x_i= s_k y_k (c_k)$.
Let $(x_1, \cdots, x_d)\in X^d$ be such a point.

Let $s\in F$ and $i\in \{1, \cdots, d\}$. Once $i= \sigma_{s_k} (c_k)$ for some $s_k\in F_k$ and $c_k\in C_k, k\in \{1, \cdots, l\}$, if $s s_k\in F_k$, then $s x_i= s s_k y_k (c_k)= x_{\sigma_{s s_k} (c_k)}= x_{\sigma_s \sigma_{s_k} (c_k)}= x_{\sigma_s (i)}$. Which implies that
\begin{equation} \label{1107120038}
\frac{1}{d} \sum_{i= 1}^d \rho^2 (s x_i, x_{\sigma_s (i)})= \frac{1}{d} \sum_{i\in \{1, \cdots, d\}\setminus E} \rho^2 (s x_i, x_{\sigma_s (i)})\le \frac{M^2}{d} |\{1, \cdots, d\}\setminus E|,
\end{equation}
where
$$E= \bigcup_{k= 1}^l \sigma (s^{- 1} F_k\cap F_k) C_k.$$
Using the construction of $C_1, \cdots, C_l$ again, by \eqref{1107120045} one has
\begin{equation} \label{1107120043}
|E|= \sum_{k= 1}^l |s^{- 1} F_k\cap F_k|\cdot |C_k|\ge (1- \delta') \sum_{k= 1}^l |F_k|\cdot |C_k|\ge d (1- 2 \delta').
\end{equation}
Combining \eqref{1107120038} with \eqref{1107120043}, we obtain
$$\frac{1}{d} \sum_{i= 1}^d \rho^2 (s x_i, x_{\sigma_s (i)})\le 2 \delta' M^2.$$
In particular, $(x_1, \cdots, x_d)\in X^d_{F, \delta, \sigma}$ follows from the selection of $\delta'$.
Now assume $(x_1, \cdots, x_d)\in U_1\times \cdots\times U_d$ for some $U_1, \cdots, U_d\in \mathcal{U}$. For each $k\in \{1, \cdots, l\}$, and any $s_k\in F_k, c_k\in C_k$, $y_k (c_k)= s_k^{- 1} x_{\sigma_{s_k} (c_k)}\in s_k^{- 1} U_{\sigma_{s_k} (c_k)}$, and so $y_k (c_k)$ is contained in the element $\bigcap\limits_{s_k\in F_k} s_k^{- 1} U_{\sigma_{s_k} (c_k)}$ of $\mathcal{U}_{F_k}$. Thus $(y_1, \cdots, y_l)$ is contained in the element $\prod\limits_{k= 1}^l \prod\limits_{c_k\in C_k} \bigcap\limits_{s_k\in F_k} s_k^{- 1} U_{\sigma_{s_k} (c_k)}$ of $\prod\limits_{k= 1}^l (\mathcal{U}_{F_k})^{C_k}$.

From the above discussions one has
\begin{eqnarray} \label{1107120112}
\log N (\mathcal{U}^d, X^d_{F, \delta, \sigma})&\ge & \log N \left(\prod\limits_{k= 1}^l (\mathcal{U}_{F_k})^{C_k}, \prod\limits_{k= 1}^l X^{C_k}\right)\nonumber \\
&= & \sum_{k= 1}^l |C_k| \log N (\mathcal{U}_{F_k}, X)\nonumber \\
&\ge & \sum_{k= 1}^l |C_k|\cdot |F_k| (h^a (G, \mathcal{U})- \theta)\ (\text{using \eqref{1107112321}})\nonumber \\
&\ge & \sum_{k= 1}^l |C_k|\cdot |F_k| h^a (G, \mathcal{U})- d \theta\nonumber \\
&\ge & d (1- \delta') h^a (G, \mathcal{U})- d \theta.
\end{eqnarray}
Then \eqref{1107111631} follows from \eqref{1107111645} and \eqref{1107120112}.
\end{proof}

Theorem \ref{1107101454} follows from Lemma \ref{1107120120} and Lemma \ref{1107111610}.

Now let's turn to the proof of Theorem \ref{1107101633}.

Let $\nu\in \mathcal{M} (X)$ and $\mathcal{V}\in \mathcal{C}_X, 0< a< 1, F\in \mathcal{F}_G$. Set
$$b_\nu (F, a, \mathcal{V})= \min
\{|\mathcal{C}|: \mathcal{C}\subseteq \mathcal{V}_F\ \text{and}\ \nu (\cup \mathcal{C})\ge a\}.$$

Inspired by \cite[Lemma 5.11]{We} it is not hard to obtain \cite[Lemma 4.15]{HYZ}.

\begin{lem} \label{1107142100}
Let $\nu\in \mathcal{M} (X)$ and $\mathcal{V}\in \mathcal{C}_X, 0< a< 1, F\in \mathcal{F}_G$. Then
\begin{equation*}
 H_\nu (\mathcal{V}_F)\le \log b_\nu (F, a, \mathcal{V})+ (1- a) |F| \log N (\mathcal{V}, X)+ \log 2.
\end{equation*}
\end{lem}

Observe that by \cite[Page 204 and Theorem 4.2]{Vara} there exists a surjective Borel map $X\rightarrow \mathcal{M}^e (X, G), x\mapsto \mu_x$ such that
\begin{enumerate}

\item $\mu_{s x}= \mu_x$ for all $x\in X$ and $s\in G$;

\item for each $\nu\in \mathcal{M}^e (X, G)$, $\nu$ is the unique $\mu\in \mathcal{M} (X, G)$ with $\mu (X_\nu)= 1$, where $X_\nu= \{x\in X: \mu_x= \nu\}$; and

\item for every $\mu\in \mathcal{M} (X, G)$ and $A\in \mathcal{B}_X$ one has $\mu (A)= \int_X \mu_x (A) d \mu (x)$.
\end{enumerate}
Furthermore, it is essentially unique in the sense that if $x\mapsto \mu_x'$ is another map satisfying the same properties then there exists $B\in \mathcal{B}_{X, G}$ such that $\mu (B)= 0$ for every $\mu\in \mathcal{M} (X, G)$ and $\mu_x= \mu_x'$ for each $x\in X\setminus B$, where
$$\mathcal{B}_{X, G}= \{A\in \mathcal{B}_X: s A= A\ \text{for all}\ s\in G\}.$$
Let $\mu\in \mathcal{M} (X, G)$. Then $\mu= \int_X \mu_x d \mu (x)$ is the \emph{ergodic decomposition} of $\mu$ (from now on we will fix it without any special statement) and $\E_\mu (f| \mathcal{B}_{X, G}) (x)= \int_X f d \mu_x$ for $\mu$-a.e. $x\in X$ once $f$ is a real-valued bounded Borel measurable function over $X$, where $\E_\mu (f| \mathcal{B}_{X, G})$ denotes the $\mu$-conditional expectation of $f$ relative to $\mathcal{B}_{X, G}$. In particular, if $B\in \mathcal{B}_{X, G}$ then $\mu_x (B)= 1$ for $\mu$-a.e. $x\in B$. Moreover, if $G$ is infinite then for each $\mathcal{V}\in \mathcal{C}_X$ one has \cite[Lemma 3.12]{HYZ}:
\begin{equation} \label{1107142149}
h_\mu^a (G, \mathcal{V})= \int_X h_{\mu_x}^a (G, \mathcal{V}) d \mu (x).
\end{equation}
If $G$ is finite then it is easy to see that for each $\mathcal{V}\in \mathcal{C}_X$ one has
\begin{equation*} \label{1107282211}
h_\mu^a (G, \mathcal{V})= \inf_{\alpha\in \mathcal{P}_X, \alpha\succeq \mathcal{V}_G} \frac{1}{|G|} H_\mu (\alpha).
\end{equation*}

Let's recall the following result from \cite{HYZ} (see \cite[Lemma 3.6]{HYZ} and \cite[Proposition 3.9]{HYZ} for the case that $G$ is finite and infinite, respectively).

\begin{lem} \label{1107151941}
Let $\mathcal{U}\in \mathcal{C}_X^o$. Then the bounded function $\mathcal{M} (X, G)\ni \mu\mapsto h_\mu^a (G, \mathcal{U})$ is Borel measurable. \end{lem}

Now following the ideas of Lemma \ref{1107111610} and \cite[Lemma 6.4]{KLAJM} let us prove:

\begin{lem} \label{1107131641}
Let $\mathcal{U}\in \mathcal{C}_X^o, \mu\in \mathcal{M} (X, G)$ and $\delta> 0, L\in \mathcal{F}_{C (X)}, F\in \mathcal{F}_G$. Then
$h_{F, \delta, \mu, L} (G, \mathcal{U})\ge \int_X h_{\mu_x}^a (G, \mathcal{U}) d \mu (x)$.
\end{lem}
\begin{proof}
Let $1> \epsilon> 0$. We are to prove $h_{F, \delta, \mu, L} (G, \mathcal{U})\ge \int_X h_{\mu_x}^a (G, \mathcal{U}) d \mu (x)- \epsilon$.

Let $\kappa> 0$ such that $\kappa (2+ |\mathcal{U}|)\le \frac{\epsilon}{2}$ and $\kappa\le \frac{1}{2 |G|}$ if additionally $G$ is finite.

As all $x\mapsto \mu_x (f), f\in L$ and $x\mapsto h_{\mu_x}^a (G, \mathcal{U})$ are bounded $\mathcal{B}_{X, G}$-measurable function over $X$ (using Lemma \ref{1107151941}), there exists $\widetilde{B}\in \mathcal{P}_X$ such that $\widetilde{B}\subseteq \mathcal{B}_{X, G}$ and
\begin{equation} \label{1107142211}
\max_{B\in \widetilde{B}} \max_{f\in L} \left(\sup_{x\in B} \mu_x (f)- \inf_{x\in B} \mu_x (f)\right)< \frac{\delta}{8},
\end{equation}
\begin{equation} \label{1107152016}
\sum_{B\in \widetilde{B}} \mu (B) \inf_{x\in B} h_{\mu_x}^a (G, \mathcal{U})\ge \int_X h_{\mu_x}^a (G, \mathcal{U}) d \mu (x)- \kappa.
\end{equation}

Denote by $\mathcal{B}$ the set of all atoms of $\widetilde{B}$ with positive $\mu$-measure and set
$\tau= \frac{1}{2} \min\limits_{B\in \mathcal{B}} \mu (B)> 0$.
Observe that for each $x\in X$, as $F'\in \mathcal{F}_G$ becomes more and more two-sided invariant, $\frac{1}{|F'|} H_{\mu_x} (\mathcal{U}_{F'})$ converges to $h_{\mu_x}^a (G, \mathcal{U})$. In particular, once $F'\in \mathcal{F}_G$ is sufficiently two-sided invariant then $\mu (X (F'))\ge 1- \frac{\tau}{2}$, where
\begin{equation*} \label{1107160152}
X (F')= \left\{x\in X: \frac{1}{|F'|} H_{\mu_x} (\mathcal{U}_{F'})\ge h_{\mu_x}^a (G, \mathcal{U})- \kappa\right\}.
\end{equation*}
The measurability of $X (F')$ is easy to check, for example using \eqref{1107160141}.

By the mean ergodic theorem \cite[Theorem 2.1]{We} (see also \cite[Page 44]{MO}) for each $f\in L$, as $F'\in \mathcal{F}_G$ becomes more and more two-sided invariant, $\frac{1}{|F'|} \sum\limits_{s\in F'} f\circ s$ converges to $\E_\mu (f| \mathcal{B}_{X, G})$ in the sense of $L^2$, no matter if $G$ is infinite. In particular, once $F'\in \mathcal{F}_G$ is sufficiently two-sided invariant then there exists $W_{F'}\in \mathcal{B}_X$
with
\begin{equation*} \label{1107142245}
\mu (W_{F'})> 1- \tau \kappa\ \text{and}\ \sup_{x\in W_{F'}} \max_{f\in L} |\frac{1}{|F'|} \sum_{s\in F'} f (s x)- \mu_x (f)|< \frac{\delta}{8}.
\end{equation*}
For each $B\in \mathcal{B}$, as $B\in \mathcal{B}_{X, G}$, $\mu_x (B)= 1$ for $\mu$-a.e. $x\in B$ and
\begin{eqnarray*} \label{1107152154}
& & \hskip -26pt \mu (B)- \tau \kappa< \mu (W_{F'}\cap B)\nonumber \\
&= & \int_B \E_\mu (1_{W_{F'}}| \mathcal{B}_{X, G}) (x) d \mu (x)= \int_B \mu_x (W_{F'}) d \mu (x)\nonumber \\
&\le & (1- \kappa) \mu (\{x\in B: \mu_x (W_{F'})\le 1- \kappa\})+ \mu (\{x\in B: \mu_x (W_{F'})> 1- \kappa\})\nonumber \\
&= & (1- \kappa) \mu (B)+ \kappa \mu (\{x\in B: \mu_x (W_{F'})> 1- \kappa\}),
\end{eqnarray*}
by the selection of $\tau$, it is easy to check
$$\mu (\{x\in B: \mu_x (W_{F'})> 1- \kappa\})> \mu (B)- \tau\ge \tau,$$
 and so, there exists $x\in B\cap X (F')$ such that $\mu_x (W_{F'}\cap B)> 1- \kappa$.

Let $\delta'> 0$ such that the diameter of $X$ is at most $\sqrt{\frac{\delta^2}{2 \delta'}}$ and
\begin{equation} \label{1107150124}
\delta'< \tau,
\delta' |\mathcal{B}| \max\limits_{f\in L} ||f||< \frac{\delta}{4}, |\mathcal{B}| \delta' |\mathcal{U}|< \frac{\epsilon}{4}.
\end{equation}
Let $M\in \N$ be large enough so that $\frac{|\mathcal{B}|}{M}< \delta'$.
Let $\delta''> 0$ such that $2 \delta''< \delta'$ and
\begin{equation} \label{1107150158}
4 \delta'' \max\limits_{f\in L} ||f||< \frac{\delta}{2}, 2 \delta'' |\mathcal{U}|+ \kappa< \frac{\epsilon}{4}.
\end{equation}

By the above discussions and Lemma \ref{1107112242}, there are an $l\in \N$ and $F_1, \cdots, F_l\in \mathcal{F}_G$ which are sufficiently two-sided invariant so that
\begin{equation} \label{1107142353}
\min_{k= 1}^l \min_{s\in F} \frac{|s^{- 1} F_k\cap F_k|}{|F_k|}\ge 1- \delta'
\end{equation}
and
for each $k= 1, \cdots, l$ there exists $W_{F_k}\in \mathcal{B}_X$ and $x (k, B)\in B\cap X (F_k)$ satisfying
\begin{equation} \label{1107142332}
\mu_{x (k, B)} (W_{F_k}\cap B)> 1- \kappa
\end{equation}
and
\begin{equation} \label{1107142329}
\max_{k= 1}^l \sup_{x\in W_{F_k}} \max_{f\in L} |\frac{1}{|F_k|} \sum_{s\in F_k} f (s x)- \mu_x (f)|< \frac{\delta}{8},
\end{equation}
additionally, if $G$ is infinite then we also require
\begin{equation} \label{1107160241}
\max_{k= 1}^l \frac{\log 2}{|F_k|}< \kappa,
\end{equation}
such that once $\sigma: G\rightarrow Sym (d)$ is a good enough sofic approximation for $G$ with some $d\in \N$ then there exist $C_1, \cdots, C_l\subseteq \{1, \cdots, d\}$ satisfying
\begin{enumerate}

\item the sets $\sigma (F_k) C_k, k\in \{1, \cdots, l\}$ are pairwise disjoint;

\item $\{\sigma (F_k) C_k: k\in \{1, \cdots, l\}\}$ $(1- \delta'')$-covers $\{1, \cdots, d\}$;

\item for every $k\in \{1, \cdots, l\}$, the map $F_k\times C_k\ni (s, c)\mapsto \sigma_s (c)$ is bijective; and

\item for all $k\in \{1, \cdots, l\}$ and $s\in F, s_k\in F_k, c_k\in C_k$, $\sigma_{s s_k} (c_k)= \sigma_s \sigma_{s_k} (c_k)$.
\end{enumerate}
The existence of such subsets $F_1, \cdots, F_l$ is ensured by the amenability of $G$.
For each $k= 1, \cdots, l$ and any $B\in \mathcal{B}$, as $x (k, B)\in B\cap X (F_k)$, one has
\begin{equation} \label{1107160247}
\frac{1}{|F_k|} H_{\mu_{x (k, B)}} (\mathcal{U}_{F_k})\ge \inf_{x\in B} h_{\mu_x}^a (G, \mathcal{U})- \kappa.
\end{equation}

Now assume that $\sigma: G\rightarrow Sym (d)$ is a good enough sofic approximation for $G$ with some $d\in \N$, and so $d$ is large enough satisfying
\begin{equation} \label{1107150005}
M \sum_{k= 1}^l |F_k|\le \delta'' d,
\end{equation}
and let $C_1, \cdots, C_l\subseteq \{1, \cdots, d\}$ be constructed as above. Set $\Lambda= \{k\in \{1, \cdots, l\}: |C_k|\ge M\}$ and $D= \cup \{\sigma (F_k) C_k: k\in \Lambda\}$. Using \eqref{1107150005} one has $|D|\ge (1- 2 \delta'') d$. Moreover, by the construction of $M\in \N$ for each $k\in \Lambda$ there exists a partition $\{C_{k, B}: B\in \mathcal{B}\}$ of $C_k$ such that
\begin{equation} \label{1107142340}
\max_{k\in \Lambda} \max_{B\in \mathcal{B}} |\frac{|C_{k, B}|}{|C_k|}- \mu (B)|< \delta'.
\end{equation}

For each $k\in \Lambda$ let $y_k\in \prod\limits_{B\in \mathcal{B}} (W_{F_k}\cap B)^{C_{k, B}}$. From the construction of $C_1, \cdots, C_l$ it is not hard to see that there exists at least one point $(x_1, \cdots, x_d)\in X^d$ such that once $i\in \sigma (F_k) C_{k, B}$ for some $k\in \Lambda$ and $B\in \mathcal{B}$, say $i= \sigma_{s_k} (c_{k, B})$ with $s_k\in F_k$ and $c_{k, B}\in C_{k, B}$, then $x_i= s_k y_k (c_{k, B})$. Let $(x_1, \cdots, x_d)\in X^d$ be such a point.

Let $f\in L$. Let $k\in \Lambda$ and $B\in \mathcal{B}$. As $B\in \mathcal{B}_{X, G}$,
\begin{equation} \label{1107150044}
\int_B f d \mu= \int_B \E_\mu (f| \mathcal{B}_{X, G}) d \mu= \int_B \mu_x (f) d \mu (x).
\end{equation}
For each $c_{k, B}\in C_{k, B}$, observe $y_k (c_{k, B})\in W_{F_k}\cap B$, one has
\begin{eqnarray*} \label{1107150053}
& & |\frac{1}{|F_k|} \sum_{s_k\in F_k} f (x_{\sigma_{s_k} (c_{k, B})})- \frac{1}{\mu (B)} \int_B f d \mu|\nonumber \\
&\le & |\frac{1}{|F_k|} \sum_{s_k\in F_k} f (s_k y_k (c_{k, B}))- \mu_{y_k (c_{k, B})} (f)|+ |\mu_{y_k (c_{k, B})} (f)- \frac{1}{\mu (B)} \int_B f d \mu|\nonumber \\
&< & \frac{\delta}{8}+ \frac{\delta}{8}\ (\text{using \eqref{1107142211}, \eqref{1107142329} and \eqref{1107150044}})= \frac{\delta}{4}.
\end{eqnarray*}
Summing over all $c_{k, B}\in C_{k, B}$ we obtain
\begin{equation} \label{1107150102}
|\frac{1}{|\sigma (F_k) C_{k, B}|} \sum_{i\in \sigma (F_k) C_{k, B}} f (x_i) - \frac{1}{\mu (B)} \int_B f d \mu|< \frac{\delta}{4}.
\end{equation}
Thus
\begin{eqnarray*} \label{1107150111}
& & |\frac{1}{|\sigma (F_k) C_k|} \sum_{i\in \sigma (F_k) C_k} f (x_i)- \mu (f)|\nonumber \\
&\le & |\sum_{B\in \mathcal{B}} \frac{1}{|\sigma (F_k) C_k|} \sum_{i\in \sigma (F_k) C_{k, B}} f (x_i)- \sum_{B\in \mathcal{B}} \frac{|\sigma (F_k) C_{k, B}|}{|\sigma (F_k) C_k|}\cdot \frac{1}{\mu (B)} \int_B f d \mu|+ \nonumber \\
& & |\sum_{B\in \mathcal{B}} \frac{|\sigma (F_k) C_{k, B}|}{|\sigma (F_k) C_k|}\cdot \frac{1}{\mu (B)} \int_B f d \mu- \sum_{B\in \mathcal{B}} \mu (B) \cdot \frac{1}{\mu (B)} \int_B f d \mu|\nonumber \\
&< & \frac{\delta}{4}+ \sum_{B\in \mathcal{B}} \frac{\delta'}{\mu (B)} |\int_B f d \mu|\ (\text{using \eqref{1107142340} and \eqref{1107150102}})\nonumber \\
&\le & \frac{\delta}{4}+ \delta' |\mathcal{B}|\cdot ||f||< \frac{\delta}{2}\ (\text{using \eqref{1107150124}}).
\end{eqnarray*}
By the construction of $C_1, \cdots, C_l$, summing over all $k\in \Lambda$ we obtain
\begin{equation} \label{1107150128}
|\frac{1}{|D|} \sum_{i\in D} f (x_i)- \mu (f)|< \frac{\delta}{2},
\end{equation}
and hence
\begin{eqnarray} \label{1107150203}
\hskip 16pt & & |\frac{1}{d} \sum_{i= 1}^d f (x_i)- \mu (f)|\nonumber \\
\hskip 16pt &\le & |\frac{1}{d} \sum_{i\in \{1, \cdots, d\}\setminus D} f (x_i)|+ \left(\frac{1}{|D|}- \frac{1}{d}\right)|\sum_{i\in D} f (x_i)|+ |\frac{1}{|D|} \sum_{i\in D} f (x_i)- \mu (f)|\nonumber \\
\hskip 16pt &\le & ||f|| \left(\frac{|\{1, \cdots, d\}\setminus D|}{d}+ \frac{d- |D|}{d}\right)+ \frac{\delta}{2}\ (\text{using \eqref{1107150128}})\nonumber \\
\hskip 16pt &\le & 4 \delta'' ||f||+ \frac{\delta}{2}\ (\text{as}\ |D|\ge (1- 2 \delta'') d)< \delta\ (\text{using \eqref{1107150158}}).
\end{eqnarray}

Let $s\in F$ and $i\in \{1, \cdots, d\}$. Once $i= \sigma_{s_k} (c_{k, B})$ with some $s_k\in F_k$ and $c_{k, B}\in C_{k, B}, k\in \Lambda, B\in \mathcal{B}$, if $s s_k\in F_k$ then $s x_i= s s_k y_k (c_{k, B})= x_{\sigma_{s s_k} (c_{k, B})}= x_{\sigma_s (i)}$. That is, $s x_i= x_{\sigma_s (i)}$ for each $i\in E$, where
$$E= \bigcup_{k\in \Lambda} \bigcup_{B\in \mathcal{B}} \sigma (s^{- 1} F_k\cap F_k) C_{k, B}= \bigcup_{k\in \Lambda} \sigma (s^{- 1} F_k\cap F_k) C_k.$$
Then by the construction of $C_1, \cdots, C_l$ one has
\begin{eqnarray} \label{1107150007}
|E|&= & \sum_{k\in \Lambda} |s^{- 1} F_k\cap F_k|\cdot |C_k|\ge (1- \delta') \sum_{k\in \Lambda} |F_k|\cdot |C_k|\ (\text{using \eqref{1107142353}})\nonumber \\
&\ge & (1- \delta') \sum_{k= 1}^l |F_k|\cdot |C_k|- d \delta''\ (\text{using \eqref{1107150005}})\nonumber \\
&\ge & d ((1- \delta') (1- \delta'')- \delta'')\ge d (1- \delta'- 2 \delta'').
\end{eqnarray}
Moreover, by the selection of $\delta'$ and $\delta''$ one has
\begin{eqnarray*} \label{1107150014}
\frac{1}{d} \sum_{i= 1}^d \rho^2 (s x_i, x_{\sigma_s (i)})&= & \frac{1}{d} \sum_{i\in \{1, \cdots, d\}\setminus E} \rho^2 (s x_i, x_{\sigma_s (i)})\nonumber \\
&\le & \frac{1}{d} |\{1, \cdots, d\}\setminus E|\cdot \frac{\delta^2}{2 \delta'}\nonumber \\
&\le & (\delta'+ 2 \delta'')\cdot \frac{\delta^2}{2 \delta'}\ (\text{using \eqref{1107150007}})< \delta^2.
\end{eqnarray*}
Combined with \eqref{1107150203}, we obtain $(x_1, \cdots, x_d)\in X^d_{F, \delta, \sigma, \mu, L}$.

Now, if $(x_1, \cdots, x_d)\in U_1\times \cdots\times U_d$ for some $U_1, \cdots, U_d\in \mathcal{U}$. For each $k\in \Lambda$ and any $s_k\in F_k, c_{k, B}\in C_{k, B}$ with $B\in \mathcal{B}$, $y_k (c_{k, B})= s_k^{- 1} x_{\sigma_{s_k} (c_{k, B})}\in s_k^{- 1} U_{\sigma_{s_k} (c_{k, B})}$, and so $y_k (c_{k, B})$ is contained in the element of $\bigcap\limits_{s_k\in F_k} s_k^{- 1} U_{\sigma_{s_k} (c_{k, B})}$ of $\mathcal{U}_{F_k}$. Thus, $\prod\limits_{k\in \Lambda} y_k$ is contained in the element $\prod\limits_{k\in \Lambda} \prod\limits_{B\in \mathcal{B}} \prod\limits_{c_{k, B}\in C_{k, B}} \bigcap\limits_{s_k\in F_k} s_k^{- 1} U_{\sigma_{s_k} (c_{k, B})}$ of $\prod\limits_{k\in \Lambda} \prod\limits_{B\in \mathcal{B}} (\mathcal{U}_{F_k})^{C_{k, B}}$. From this, we obtain readily
\begin{eqnarray} \label{1107150240}
\log N (\mathcal{U}^d, X^d_{F, \delta, \sigma, \mu, L})&\ge & \log N \left(\prod\limits_{k\in \Lambda} \prod_{B\in \mathcal{B}} (\mathcal{U}_{F_k})^{C_{k, B}}, \prod_{k\in \Lambda} \prod\limits_{B\in \mathcal{B}} (W_{F_k}\cap B)^{C_{k, B}}\right)\nonumber \\
&= & \sum_{k\in \Lambda} \sum_{B\in \mathcal{B}} |C_{k, B}| \log N (\mathcal{U}_{F_k}, W_{F_k}\cap B)\nonumber \\
&\ge & \sum_{k\in \Lambda} \sum_{B\in \mathcal{B}} |C_{k, B}| \log b_{\mu_{x (k, B)}} (F_k, 1- \kappa, \mathcal{U})\ (\text{using \eqref{1107142332}}).
\end{eqnarray}

If $G$ is finite. Observe that if $x\in X$ say $G x= \{x_1, \cdots, x_p\}$ with $p\le |G|$ then
$\frac{1}{p} \sum\limits_{q= 1}^p \delta_{x_q}\in \mathcal{M}^e (X, G)$;
in fact, each element of $\mathcal{M}^e (X, G)$ must be of such form. As $\kappa\le \frac{1}{2 |G|}$,
for $k\in \Lambda$ and $B\in \mathcal{B}$, once a Borel measurable subset has $\mu_{x (k, B)}$-measure at least $1- \kappa$ then it has $\mu_{x (k, B)}$-measure 1, which implies readily $H_{\mu_{x (k, B)}} (\mathcal{U}_{F_k})\le \log b_{\mu_{x (k, B)}} (F_k, 1- \kappa, \mathcal{U})$, using \eqref{1107160247} and \eqref{1107150240} one has
\begin{equation} \label{1107160256}
\log N (\mathcal{U}^d, X^d_{F, \delta, \sigma, \mu, L})\ge \sum_{k\in \Lambda} \sum_{B\in \mathcal{B}} |C_{k, B}|\cdot |F_k| \left(\inf_{x\in B} h^a_{\mu_x} (G, \mathcal{U})- \kappa\right).
\end{equation}

If $G$ is infinite. Using Lemma \ref{1107142100} and \eqref{1107160241}, \eqref{1107160247}, \eqref{1107150240} we obtain
\begin{eqnarray} \label{1107160257}
\log N (\mathcal{U}^d, X^d_{F, \delta, \sigma, \mu, L})&\ge & \sum_{k\in \Lambda} \sum_{B\in \mathcal{B}} |C_{k, B}| (H_{\mu_{x (k, B)}} (\mathcal{U}_{F_k})- \kappa |F_k|\cdot |\mathcal{U}|- \log 2)\nonumber \\
&\ge & \sum_{k\in \Lambda} \sum_{B\in \mathcal{B}} |C_{k, B}|\cdot |F_k| \left(\inf_{x\in B} h^a_{\mu_x} (G, \mathcal{U})- 2 \kappa- \kappa |\mathcal{U}|\right).
\end{eqnarray}

As $\kappa (2+ |\mathcal{U}|)\le \frac{\epsilon}{2}$, combining \eqref{1107160256} and \eqref{1107160257} one has (no matter if $G$ is finite)
\begin{eqnarray*}
& & \log N (\mathcal{U}^d, X^d_{F, \delta, \sigma, \mu, L}) \\
&\ge & \sum_{k\in \Lambda} \sum_{B\in \mathcal{B}} |C_{k, B}|\cdot |F_k| \left(\inf_{x\in B} h^a_{\mu_x} (G, \mathcal{U})- \frac{\epsilon}{2}\right) \\
&\ge & \sum_{k\in \Lambda} \sum_{B\in \mathcal{B}} |C_{k, B}|\cdot |F_k| \inf_{x\in B} h^a_{\mu_x} (G, \mathcal{U})- \frac{d \epsilon}{2} \\
&\ge & \sum_{k\in \Lambda} |C_k|\cdot |F_k| \sum_{B\in \mathcal{B}} (\mu (B)- \delta') \inf_{x\in B} h^a_{\mu_x} (G, \mathcal{U})- \frac{d \epsilon}{2}\ (\text{using \eqref{1107142340}}) \\
&\ge & d (1- 2 \delta'') \left(\int_X h_{\mu_x}^a (G, \mathcal{U}) d \mu (x)- \kappa\right)- d |\mathcal{B}| \delta' |\mathcal{U}|- \frac{d \epsilon}{2} \\
& & (\text{using \eqref{1107152016} and the fact of}\ |D|\ge (1- 2 \delta'') d) \\
&\ge & d \int_X h_{\mu_x}^a (G, \mathcal{U}) d \mu (x)- d \left(\kappa+ 2 \delta'' |\mathcal{U}|+ |\mathcal{B}| \delta' |\mathcal{U}|+ \frac{\epsilon}{2}\right) \\
&\ge & d \int_X h_{\mu_x}^a (G, \mathcal{U}) d \mu (x)- d \epsilon\ (\text{by the selection of $\delta', \delta''$}).
\end{eqnarray*}
The conclusion follows easily from the above estimation.
\end{proof}

Using \eqref{1107142149} and \cite[Lemma 6.1]{KLAJM}, following the proof of \cite[Proposition 4.18]{HYZ} we may obtain the following result.

\begin{lem} \label{1107211626}
Let $\mu\in \mathcal{M} (X, G), \mathcal{U}\in \mathcal{C}_X$ and $\epsilon> 0, 0< a< 1$. Assume that $G$ is infinite. Then, once $F\in \mathcal{F}_G$ is sufficiently left invariant,
$$\frac{1}{|F|} \log b_\mu (F, a, \mathcal{U})\le h_\mu^a (G, \mathcal{U})+ \epsilon.$$
\end{lem}

Remark that, \cite[Proposition 4.18]{HYZ} considered the case of $\mu\in \mathcal{M}^e (X, G)$, and with the help of \eqref{1107142149} and \cite[Lemma 6.1]{KLAJM} it can be generalized to all $\mu\in \mathcal{M} (X, G)$ almost by the same proof of \cite[Proposition 4.18]{HYZ}.

Similar to Lemma \ref{1107122331}, we could prove:

\begin{lem} \label{1107222344}
Let $F\in \mathcal{F}_G$ and $\mathcal{U}\in \mathcal{C}_X^o$. Assume that $\emptyset\neq K_F\subseteq X$ is a closed subset and $\mathcal{V} (F)\subseteq \mathcal{U}_F$ satisfies $\cup \mathcal{V} (F)\supseteq K_F$. Then there exists $\delta> 0$ such that
$$K_{F, \delta}= \left\{(x_s)_{s\in F}\in X^F: \max_{s\in F} \rho (x_s, s x)< \delta\ \text{for some}\ x\in K_F\right\}$$
can be covered by at most $|\mathcal{V} (F)|$ elements of $\mathcal{U}^F$.
\end{lem}

Now following the ideas of Lemma \ref{1107120120} and \cite[Lemma 6.3]{KLAJM}, let us prove:

\begin{lem} \label{1107131627}
Let $\mathcal{U}\in \mathcal{C}_X^o$ and $\mu\in \mathcal{M} (X, G), \kappa> 0$.
Assume that $G$ is infinite. Then there exist $F\in \mathcal{F}_G, \delta> 0$ and $L\in \mathcal{F}_{C (X)}$ such that
$$h_{F, \delta, \mu, L} (G, \mathcal{U})\le \int_X h^a_{\mu_x} (G, \mathcal{U}) d \mu (x)+ 6 \kappa.$$
\end{lem}
\begin{proof}
By Lemma \ref{1107151941}, there exists $\mathcal{R}'\in \mathcal{P}_X$ such that $\mathcal{R}'\subseteq \mathcal{B}_{X, G}$ and
\begin{equation} \label{1107211718}
\max_{R\in \mathcal{R}'} \left(\sup_{x\in R} h^a_{\mu_x} (G, \mathcal{U})- \inf_{x\in R} h^a_{\mu_x} (G, \mathcal{U})\right)< \kappa.
\end{equation}
Denote by $\mathcal{R}$ the set of all atoms from $\mathcal{R}'$ with positive $\mu$-measure. Thus
\begin{equation} \label{1107230111}
\int_X h^a_{\mu_x} (G, \mathcal{U}) d \mu (x)\ge \sum_{R\in \mathcal{R}} \mu (R) \xi_R- \kappa\ (\text{using \eqref{1107211718}}),
\end{equation}
where $\xi_R= \sup\limits_{x\in R} h^a_{\mu_x} (G, \mathcal{U})\le \log |\mathcal{U}|$ for each $R\in \mathcal{R}$.

Using \cite[Lemma 6.2]{KLAJM}, there exist $M'\in \N$ and $\omega: \N\rightarrow (0, 1)$ such that, for any $F'\in \mathcal{F}_G$ with $|F'|\ge M'$, once $\sigma: G\rightarrow Sym (d)$ is a good enough sofic approximation for $G$ with some $d\in \N$ then the number of $A\subseteq \{1, \cdots, d\}$ with $\max\limits_{s\in F'} |A\Delta \sigma_s (A)|\le \omega (|F'|) d$ is at most $\exp (\frac{\kappa d}{|\mathcal{R}|})$.

By Stirling's approximation formula,
there exists $\eta> 0$ small enough
such that
$$2 |\mathcal{R}| \eta \log |\mathcal{U}|\le \kappa, \frac{(\mu (R)+ 2 |\mathcal{R}| \eta)}{1- \eta}\le \frac{\mu (R) (\xi_R+ 2 \kappa)}{\xi_R+ \kappa}$$
 for each $R\in \mathcal{R}$ and, for every $R\in \mathcal{R}$ and any non-empty finite subset $\Upsilon$ the number of $\Upsilon'\subseteq \Upsilon$ with $\frac{|\Upsilon'|}{|\Upsilon|}\ge \frac{\mu (R)- \eta}{\mu (R)+ \eta}$ is at most $e^{\kappa |\Upsilon|}$.

By Lemma \ref{1107110016} there exist $l\in \N$ and $\eta'> 0$ such that in $\mathcal{F}_G$ once $e\in F_1\subseteq \cdots\subseteq F_l$ satisfies $|F_{k- 1}^{- 1} F_k\setminus F_k|\le \eta' |F_k|$ for all $k= 2, \cdots, l$, then for any good enough sofic approximation $\sigma: G\rightarrow Sym (d)$ for $G$ with some $d\in \N$ and every $Y_R\subseteq \{1, \cdots, d\}$ with $|Y_R|\ge d (\mu (R)- \eta)$ for all $R\in \mathcal{R}$, there exists, for every $R\in \mathcal{R}$, subsets $C_{R, 1}, \cdots, C_{R, l}\subseteq Y_R$ satisfying
\begin{enumerate}

\item the sets $\sigma (F_k) C_{R, k}, k\in \{1, \cdots, l\}$ are pairwise disjoint;

\item $\{\sigma (F_k) c: c\in C_{R, k}\}$ is $\eta$-disjoint for each $k= 1, \cdots, l$;

\item $\{\sigma (F_k) C_{R, k}: k\in \{1, \cdots, l\}\}$ $(\mu (R)- 2 \eta)$-covers $\{1, \cdots, d\}$; and

\item for every $k\in \{1, \cdots, l\}$ and $c\in C_{R, k}$, $F_k\ni s\mapsto \sigma_s (c)$ is bijective.
\end{enumerate}

Let $0< \tau< \frac{\eta}{4}$ satisfy $(1- 2 \tau) \mu (R)\ge \mu (R)- \frac{\eta}{2}$ for each $R\in \mathcal{R}$.
Observe that, for each $R\in \mathcal{R}$, $R\in \mathcal{B}_{X, G}$ and so $\mu_R\in \mathcal{M} (X, G)$ and $\mu_R$ is supported on $R$ where $\mu_R (\bullet)\doteq \frac{\mu (\bullet\cap R)}{\mu (R)}$.
By Lemma \ref{1107211626}, once $F'\in \mathcal{F}_G$ is sufficiently left invariant, for each $R\in \mathcal{R}$ there exist a Borel measurable subset $X_{R, F'}\subseteq R$ and $\mathcal{U}_{R, F'}\subseteq \mathcal{U}_{F'}$ such that $\mu_R (X_{R, F'})> 1- \frac{\tau}{l}, X_{R, F'}\subseteq \cup \mathcal{U}_{R, F'}$ and
\begin{equation} \label{1107212014}
\frac{1}{|F'|} \log |\mathcal{U}_{R, F'}|\le h^a_{\mu_R} (G, \mathcal{U})+ \kappa\le \xi_R+ \kappa\ (\text{applying \eqref{1107142149} to $\mu_R$}).
\end{equation}

Now in $\mathcal{F}_G$ we fix $e\in F_1\subseteq \cdots\subseteq F_l$ such that $|F_l|\ge M'$, $|F_{k- 1}^{- 1} F_k\setminus F_k|\le \eta' |F_k|$ for all $k= 2, \cdots, l$ and all $F_1, \cdots, F_l$ are sufficiently left invariant. For each $R\in \mathcal{R}$, set $X_R= \bigcap\limits_{k= 1}^l X_{R, F_k}$ and $V_R= \bigcap\limits_{k= 1}^l \cup \mathcal{U}_{R, F_k}$, then $\mu_R (X_R)> 1- \tau$ and $X_R\subseteq V_R$.

Let $\lambda> 0$ such that $\lambda< \min\{\frac{\eta}{8}, \frac{\omega (|F_l|)}{2 (2 |\mathcal{R}|+ 1)}\}$.
For each $R\in \mathcal{R}$, by the regularity of $\mu_R$, there exist closed subsets $Z_R$ and $Z_R'$ such that $Z_R\subseteq X_R\cap Z_R'\subseteq Z_R'\subseteq R, \mu_R (Z_R)> 1- \tau$ and $\mu_R (Z_R')> 1- \lambda$.
As $F_l Z_R'\subseteq R, R\in \mathcal{R}$ are pairwise disjoint, there exist pairwise disjoint open subsets $U_R\supseteq F_l Z_R'$ for each $R\in \mathcal{R}$. Let $R\in \mathcal{R}$. Recall that from the constructions $V_R$ is an open set. There exist open subsets $Z_R\subseteq B_R$ and $Z_R'\subseteq B_R'$ such that $\overline{B_R}\subseteq V_R, B_R\subseteq B_R'$ and $F_l B_R'\subseteq U_R$.
In $C (X)$ choose $0\le g_R\le h_R\le 1$ such that $g_R|_{Z_R}= 1, g_R|_{B_R^c}= 0$ and $h_R|_{Z_R'}= 1, h_R|_{(B_R')^c}= 0$.
Set $L= \{g_R, h_R: R\in \mathcal{R}\}\in \mathcal{F}_{C (X)}$.

Observe that for all $R\in \mathcal{R}$ and $k= 1, \cdots, l$, we could cover $\overline{B_R}$ by $\mathcal{U}_{R, F_k}$ (as $\overline{B_R}\subseteq V_R$).
By Lemma \ref{1107222344}, there is $\delta_2> 0$ small enough such that we could cover
$$\left\{(x_s)_{s\in F_k}: \max_{s\in F_k} \rho (x_s, s x)< \delta_2\ \text{for some}\ x\in \overline{B_R}\right\}$$
by at most $|\mathcal{U}_{R, F_k}|$ elements of $\mathcal{U}^{F_k}$ for all $R\in \mathcal{R}$ and $k= 1, \cdots, l$.
Moreover, we may select $\delta_4> 0$ small enough such that $\delta_4\le \delta_2$ and once $\rho (x', x'')\le \delta_4$ then
\begin{equation*} \label{1107212200}
\max_{R\in \mathcal{R}} \max_{s\in F_l} |h_R (s^{- 1} x')- h_R (s^{- 1} x'')|< \frac{1}{2}.
\end{equation*}

By the selection of $\tau$ and $\lambda$, there exists $\delta> 0$ small enough such that
$\delta\le \delta_4^2, |\mathcal{R}| \delta+ \lambda< \eta, (2 |\mathcal{R}|+ 1) \lambda+ |\mathcal{R}| (|\mathcal{R}|+ |F_l|) \delta\le \frac{\omega (|F_l|)}{2}$ and
$(1+ \lambda- 2 \tau) \mu (R)- 2 \lambda- (|\mathcal{R}|+ |F_l|) \delta\ge \mu (R)- \eta$ for all $R\in \mathcal{R}$.

Let $\sigma: G\rightarrow Sym (d)$ be a good enough sofic approximation for $G$ with some $d\in \N$ such that $|\Lambda|\ge d (1- \lambda)$, where $\Lambda= \{a\in \{1, \cdots, d\}: \sigma_e (a)= a\}$.

For each $(x_1, \cdots, x_d)\in X^d_{F_l, \delta, \sigma, \mu, L}$, let us consider $\Omega_R^*= \Lambda\cap \Lambda^*\cap \Omega_R''$ and $\Theta_R^*= \Lambda\cap \Lambda^*\cap \Theta_R''\subseteq \Omega_R^*$, where
$$\Lambda^*= \left\{a\in \{1, \cdots, d\}: \max\limits_{s\in F_l} \rho (x_{\sigma_s (a)}, s x_a)< \sqrt{\delta}\right\},$$
$$\Omega_R''= \left\{a\in \{1, \cdots, d\}: h_R (x_a)> \frac{1}{2}\right\}, \Omega_R'= \{a\in \{1, \cdots, d\}: h_R (x_a)> 0\},$$
$$\Theta_R''= \left\{a\in \{1, \cdots, d\}: g_R (x_a)> \frac{1}{2}\right\}, \Theta_R'= \{a\in \{1, \cdots, d\}: g_R (x_a)> 0\}.$$
Obviously, $|\Lambda^*|\ge d (1- |F_l| \delta)$. For each $a\in \Lambda^*$, as $\delta\le \delta_4^2$, one has
\begin{equation*} \label{1107212230}
\max_{R\in \mathcal{R}} \max_{s\in F_l} |h_R (x_a)- h_R (s^{- 1} x_{\sigma_s (a)})|< \frac{1}{2}.
\end{equation*}
Which implies $x_{\sigma_s (a)}\in s B_R'\subseteq U_R$ for all $a\in \Omega_R^*$ and $s\in F_l$, and so $\sigma (F_l) \Omega_R^*, R\in \mathcal{R}$ are pairwise disjoint (as $U_R, R\in \mathcal{R}$ are pairwise disjoint).

By the construction, $\Omega_R', R\in \mathcal{R}$ are pairwise disjoint. For every $R\in \mathcal{R}$, one has
\begin{equation} \label{1107212327}
\frac{|\Omega_R'|}{d}\ge \frac{1}{d} \sum_{a= 1}^d h_R (x_a)\ge \mu (h_R)- \delta\ge \mu (Z_R')- \delta\ge (1- \lambda) \mu (R)- \delta,
\end{equation}
which implies
\begin{eqnarray} \label{1107212328}
\frac{|\Omega_R^*|}{d}&\le & \frac{|\Omega_R'|}{d}\le 1- \sum_{R'\in \mathcal{R}\setminus \{R\}} \frac{|\Omega_{R'}'|}{d}\nonumber \\
&\le & (1- \lambda) \mu (R)+ |\mathcal{R}| \delta+ \lambda\ (\text{applying \eqref{1107212327} to each $R'$})\le \mu (R)+ \eta.
\end{eqnarray}
From the construction, it is easy to see
\begin{equation} \label{1107212344}
(1- \lambda) \mu (R)\le \mu (Z_R')\le \mu (h_R)\le \frac{|\Omega_R''|}{d}+ \frac{|\Omega_R'\setminus \Omega_R''|}{2 d}= \frac{|\Omega_R''|}{2 d}+ \frac{|\Omega_R'|}{2 d}.
\end{equation}
Combining \eqref{1107212328} and \eqref{1107212344} we obtain
$$\frac{|\Omega_R''|}{d}\ge (1- \lambda) \mu (R)- \lambda- |\mathcal{R}| \delta,$$
and so
\begin{equation} \label{1107221033}
\frac{|\Omega_R^*|}{d}\ge (1- \lambda) \mu (R)- 2 \lambda- (|\mathcal{R}|+ |F_l|) \delta.
\end{equation}
Observe $\sigma (F_l) \Omega_R^*\supseteq \sigma_e (\Omega_R^*)= \Omega_R^*, R\in \mathcal{R}$ are pairwise disjoint.
For every $R\in \mathcal{R}$, applying \eqref{1107221033} we obtain
\begin{equation} \label{1107221040}
\frac{|\sigma (F_l) \Omega_R^*\setminus \Omega_R^*|}{d}\le 1- \sum_{R'\in \mathcal{R}} \frac{|\Omega_{R'}^*|}{d}\le (2 |\mathcal{R}|+ 1) \lambda+ |\mathcal{R}| (|\mathcal{R}|+ |F_l|) \delta,
\end{equation}
and then
\begin{eqnarray} \label{1107221047}
\max_{s\in F_l} \frac{|\Omega_R^*\Delta \sigma_s (\Omega_R^*)|}{d}&= & \max_{s\in F_l} \left(\frac{|\Omega_R^*\setminus \sigma_s (\Omega_R^*)|}{d}+ \frac{|\sigma_s (\Omega_R^*)\setminus \Omega_R^*|}{d}\right)\nonumber \\
&= & 2 \max_{s\in F_l} \frac{|\sigma_s (\Omega_R^*)\setminus \Omega_R^*|}{d}\ (\text{observe $|\sigma_s (\Omega_R^*)|= |\Omega_R^*|$})\nonumber \\
&\le & 2 \frac{|\sigma (F_l) \Omega_R^*\setminus \Omega_R^*|}{d}\nonumber \\
&\le & 2 [(2 |\mathcal{R}|+ 1) \lambda+ |\mathcal{R}| (|\mathcal{R}|+ |F_l|) \delta]\ (\text{using \eqref{1107221040}})\le \omega (|F_l|).
\end{eqnarray}

Let $R\in \mathcal{R}$. If $a\in \Theta_R^*$ then $x_a\in B_R$. Similar to \eqref{1107212328} and \eqref{1107212344} we obtain:
\begin{equation*} \label{1107221110}
(1- \tau) \mu (R)- \delta\le \frac{|\Theta_R'|}{d}\le \frac{|\Omega_R'|}{d}\le (1- \lambda) \mu (R)+ \lambda+ |\mathcal{R}| \delta\ (\text{using \eqref{1107212328}})
\end{equation*}
and
\begin{equation*} \label{1107221115}
(1- \tau) \mu (R)\le \mu (g_R)\le \frac{|\Theta'|}{2 d}+ \frac{|\Theta_R''|}{2 d},
\end{equation*}
which implies (by the selection of $\lambda, \tau, \delta, \eta$)
\begin{eqnarray} \label{1107221117}
\frac{|\Theta_R^*|}{d}&\ge & \frac{|\Theta_R''|}{d}- \lambda- |F_l| \delta\nonumber \\
&\ge & (1+ \lambda- 2 \tau) \mu (R)- 2 \lambda- (|\mathcal{R}|+ |F_l|) \delta\ge \mu (R)- \eta.
\end{eqnarray}

Thus, by the constructions, using \eqref{1107221047} there exist $X^{d, 1}_{F_l, \delta, \sigma, \mu, L}\subseteq X^d_{F_l, \delta, \sigma, \mu, L}$ and disjoint subsets $\{\Omega_R: R\in \mathcal{R}\}$ of $\{1, \cdots, d\}$ such that for all $(x_1, \cdots, x_d)\in X^{d, 1}_{F_l, \delta, \sigma, \mu, L}$, $\Omega_R^*= \Omega_R$ for each $R\in \mathcal{R}$, and
\begin{equation} \label{1107221645}
e^{\kappa d}\cdot N (\mathcal{U}^d, X^{d, 1}_{F_l, \delta, \sigma, \mu, L})\ge N (\mathcal{U}^d, X^d_{F_l, \delta, \sigma, \mu, L}),
\end{equation}
and then using \eqref{1107212328} and \eqref{1107221117} by the selection of $\eta$ there exist $X^{d, 2}_{F_l, \delta, \sigma, \mu, L}\subseteq X^{d, 1}_{F_l, \delta, \sigma, \mu, L}$ and disjoint subsets $\{\Theta_R: R\in \mathcal{R}\}$ of $\{1, \cdots, d\}$ such that for all $(x_1, \cdots, x_d)\in X^{d, 2}_{F_l, \delta, \sigma, \mu, L}$, $\Theta_R^*= \Theta_R (\subseteq \Omega_R)$ for each $R\in \mathcal{R}$ and
$$\prod_{R\in \mathcal{R}} e^{\kappa |\Omega_R|}\cdot N (\mathcal{U}^d, X^{d, 2}_{F_l, \delta, \sigma, \mu, L})\ge N (\mathcal{U}^d, X^{d, 1}_{F_l, \delta, \sigma, \mu, L}),$$
which implies
\begin{eqnarray} \label{1107221650}
N (\mathcal{U}^d, X^d_{F_l, \delta, \sigma, \mu, L})&\le & N (\mathcal{U}^d, X^{d, 2}_{F_l, \delta, \sigma, \mu, L})\cdot \prod_{R\in \mathcal{R}} e^{\kappa |\Omega_R|}\cdot e^{\kappa d}\ (\text{using \eqref{1107221645}})\nonumber \\
&\le & N (\mathcal{U}^d, X^{d, 2}_{F_l, \delta, \sigma, \mu, L})\cdot e^{2 \kappa d}.
\end{eqnarray}

Now for each $R\in \mathcal{R}$ we construct $C_{R, 1}, \cdots, C_{R, l}\subseteq \Theta_R$ as in the beginning of the proof.
 Observe that $F_1\subseteq \cdots\subseteq F_l, \sigma (F_l) \Theta_R, R\in \mathcal{R}$ are pairwise disjoint and
 \begin{equation} \label{1107230131}
 \sum_{k= 1}^l |\sigma (F_k) C_{R, k}|\ge d (\mu (R)- 2 \eta)
 \end{equation}
 for each $R\in \mathcal{R}$, from the construction it is not hard to obtain
 \begin{equation} \label{1107230027}
 |J|\le 2 d |\mathcal{R}| \eta\ \text{where}\ J= \{1, \cdots, d\}\setminus \bigcup\limits_{R\in \mathcal{R}} \bigcup\limits_{k= 1}^l \sigma (F_k) C_{R, k},
 \end{equation}
 and
 \begin{equation} \label{1107230041}
 \sum_{k= 1}^l |F_k|\cdot |C_{R, k}|\le \frac{1}{1- \eta} \sum_{k= 1}^l |\sigma (F_k) C_{R, k}|\le \frac{d (\mu (R)+ 2 |\mathcal{R}| \eta)}{1- \eta}
 \end{equation}
 for each $R\in \mathcal{R}$ (applying \eqref{1107230131} to each $R'\in \mathcal{R}\setminus \{R\}$).

 Let $R\in \mathcal{R}, k= 1, \cdots, l$ and $c\in C_{R, k}$. For each $(x_1, \cdots, x_d)\in X^{d, 2}_{F_l, \delta, \sigma, \mu, L}$, by the selection of $\Theta_R$ we have $g_R (x_c)> \frac{1}{2}$ (and so $x_c\in B_R$) and
$$\max\limits_{s\in F_l} \rho (x_{\sigma_s (c)}, s x_c)< \sqrt{\delta}.$$
Then by the selection of $\delta$ we could cover
\begin{eqnarray*}
& & \{(x_s)_{s\in \sigma (F_k) c}: (x_1, \cdots, x_d)\in X^{d, 2}_{F_l, \delta, \sigma, \mu, L}\} \\
&\subseteq & \left\{(x_s)_{s\in \sigma (F_k) c}: \max_{s\in F_l} \rho (x_{\sigma_s (c)}, s x)< \delta_2\ \text{for some}\ x\in B_R\right\}
\end{eqnarray*}
by at most $|\mathcal{U}_{R, F_k}|$ elements of $\mathcal{U}^{\sigma (F_k) c}$, and so we could cover
$$\{(x_s)_{s\in \sigma (F_k) C_{R, k}}: (x_1, \cdots, x_d)\in X^{d, 2}_{F_l, \delta, \sigma, \mu, L}\}$$
by at most $|\mathcal{U}_{R, F_k}|^{|C_{R, k}|}$ elements of $\mathcal{U}^{\sigma (F_k) C_{R, k}}$. Thus by the selection of $\eta$ and the construction of $\mathcal{U}_{R, F_k}, R\in \mathcal{R}, k= 1, \cdots, l$ we obtain
\begin{eqnarray*} \label{1107230047}
& &\hskip -26pt \log N (\mathcal{U}^d, X^d_{F_l, \delta, \sigma, \mu, L})\nonumber \\
&\le & \log \left(\prod_{R\in \mathcal{R}} \prod_{k= 1}^l |\mathcal{U}_{R, F_k}|^{|C_{R, k}|}\cdot |\mathcal{U}|^{|J|}\cdot e^{2 \kappa d}\right)\ (\text{using \eqref{1107221650}})\nonumber \\
&\le & \sum_{R\in \mathcal{R}} \sum_{k= 1}^l |F_k| (\xi_R+ \kappa)|C_{R, k}|+ 2 d |\mathcal{R}| \eta \log |\mathcal{U}|+ 2 \kappa d\ (\text{using \eqref{1107212014} and \eqref{1107230027}})\nonumber \\
&\le & \sum_{R\in \mathcal{R}} \frac{d (\mu (R)+ 2 |\mathcal{R}| \eta)}{1- \eta}\cdot (\xi_R+ \kappa)+ 3 \kappa d\ (\text{using \eqref{1107230041}})\nonumber \\
&\le & \sum_{R\in \mathcal{R}} d\cdot \mu (R) (\xi_R+ 2 \kappa)+ 3 \kappa d\nonumber \\
&\le & d \left(\int_X h^a_{\mu_x} (G, \mathcal{U}) d \mu (x)+ 6 \kappa\right)\ (\text{using \eqref{1107230111}}).
\end{eqnarray*}
Then the conclusion follows directly from the above estimation.
\end{proof}

The following result is \cite[Lemma 3.7]{HYZ}. In fact, \cite[Lemma 3.7]{HYZ} considered the case that $G$ is infinite, whereas, the proof of it works for the case that $G$ is finite.

\begin{lem} \label{1107230225}
Let $\mu\in \mathcal{M} (X, G), M\in\N$ and $\epsilon> 0$. Then there exists $\delta> 0$ such that if $\mathcal{U}= \{U_1, \cdots, U_M\}\in \mathcal{C}_X$ and $\mathcal{V}= \{V_1, \cdots,V_M\}\in \mathcal{C}_X$ satisfy $\mu (\mathcal{U}\Delta \mathcal{V})\doteq \sum\limits_{m= 1}^M \mu (U_m\Delta V_m)< \delta$ then $|h_\mu^a (G, \mathcal{U})- h_\mu^a (G, \mathcal{V})|\le \epsilon$.
\end{lem}

Now with the help of the above lemma, let us finish the proof of Theorem \ref{1107101633}.

\begin{proof}[Proof of Theorem \ref{1107101633}]
Using \eqref{1107142149} and Lemma \ref{1107131641}, Lemma \ref{1107131627}, we obtain directly Theorem \ref{1107101633} for all $\mathcal{U}\in \mathcal{C}_X^o$, which implies $h_\mu (G, \mathcal{V})\le h_\mu^a (G, \mathcal{V})$.

Now say $\mathcal{V}= \{V_1, \cdots, V_M\}, M\in \N$. For each $\epsilon> 0$ let $\delta> 0$ be given by Lemma \ref{1107230225}. By the regularity of $\mu$, there exist a compact subset $K_m\subseteq V_m$ for each $m= 1, \cdots, M$ such that
\begin{equation} \label{1107231846}
\mu (U)< \frac{\delta}{M},\ \text{where}\ U= \bigcup_{m= 1}^M V_m\setminus K_m.
\end{equation}
For each $m= 1, \cdots, M$, set $U_m= K_m\cup U$ and $\mathcal{U}= \{U_1, \cdots, U_M\}$. It is easy to see $\mathcal{U}\in \mathcal{C}_X^o$ and $\mathcal{V}\succeq \mathcal{U}$, additionally, from \eqref{1107231846} one has $\mu (\mathcal{U}\Delta \mathcal{V})< \delta$, and so
$$h_\mu^a (G, \mathcal{V})\le h_\mu^a (G, \mathcal{U})+ \epsilon= h_\mu (G, \mathcal{U})+ \epsilon\le h_\mu (G, \mathcal{V})+ \epsilon.$$
By the arbitrariness of $\epsilon> 0$ we obtain the conclusion.
\end{proof}

\vskip 16pt

\section*{Acknowledgements}

The author would like to thank Professor Hanfeng Li for many useful discussions during the preparations of
this manuscript, thank Professor Xiangdong Ye for his constant encouragement to this topic in the past several years,
and also thank Professor Wen Huang for his comments to this preprint.
Part of this work was carried out when the author was a Post Doctoral Research Fellow at University of New South Wales, and he thanks the dynamics group there, especially Professors Anthony H. Dooley and Valentyn Ya. Golodets, for their hospitality.

The author was supported by FANEDD (No. 201018), NSFC (No.
10801035) and a grant from Chinese Ministry of Education (No. 200802461004).

\vskip 16pt

\bibliographystyle{amsplain}

\begin{thebibliography}{10}

\bibitem{AKM}
R.~L. Adler, A.~G. Konheim, and M.~H. McAndrew, \emph{Topological entropy},
  Trans. Amer. Math. Soc. \textbf{114} (1965), 309--319. \MR{0175106 (30
  \#5291)}

\bibitem{B1}
F.~Blanchard, \emph{Fully positive topological entropy and topological mixing},
  Symbolic dynamics and its applications ({N}ew {H}aven, {CT}, 1991), Contemp.
  Math., vol. 135, Amer. Math. Soc., Providence, RI, 1992, pp.~95--105.
  \MR{1185082 (93k:58134)}

\bibitem{B2}
\bysame, \emph{A disjointness theorem involving topological entropy}, Bull.
  Soc. Math. France \textbf{121} (1993), no.~4, 465--478. \MR{1254749
  (95e:54050)}

\bibitem{BGH}
F.~Blanchard, E.~Glasner, and B.~Host, \emph{A variation on the variational
  principle and applications to entropy pairs}, Ergodic Theory Dynam. Systems
  \textbf{17} (1997), no.~1, 29--43. \MR{1440766 (98k:54073)}

\bibitem{BHMMR}
F.~Blanchard, B.~Host, A.~Maass, S.~Martinez, and D.~J. Rudolph, \emph{Entropy
  pairs for a measure}, Ergodic Theory Dynam. Systems \textbf{15} (1995),
  no.~4, 621--632. \MR{1346392 (96m:28024)}

\bibitem{BowenLETDS11}
L.~Bowen, \emph{Sofic entropy and amenable groups}, Ergodic Theory Dynam.
  Systems, to appear.

\bibitem{BowenLJAMS10}
\bysame, \emph{Measure conjugacy invariants for actions of countable sofic
  groups}, J. Amer. Math. Soc. \textbf{23} (2010), no.~1, 217--245. \MR{2552252
  (2011b:37010)}

\bibitem{DZRDS}
A.~H. Dooley and G.~H. Zhang, \emph{Local entropy theory of a random dynamical
  system}, preprint (2011). \MR{arXiv: 1106.0150v2}

\bibitem{DYZ}
D.~Dou, X.~Ye, and G.~H. Zhang, \emph{Entropy sequences and maximal entropy
  sets}, Nonlinearity \textbf{19} (2006), no.~1, 53--74. \MR{2191619
  (2006i:37037)}

\bibitem{Elek}
G.~Elek, \emph{The strong approximation conjecture holds for amenable groups},
  J. Funct. Anal. \textbf{239} (2006), no.~1, 345--355. \MR{2258227
  (2007m:43001)}

\bibitem{Folner}
E.~F{\o}lner, \emph{On groups with full {B}anach mean value}, Math. Scand.
  \textbf{3} (1955), 243--254. \MR{0079220 (18,51f)}

\bibitem{GW4}
E.~Glasner and B.~Weiss, \emph{On the interplay between measurable and
  topological dynamics}, Handbook of dynamical systems. {V}ol. 1{B}, Elsevier
  B. V., Amsterdam, 2006, pp.~597--648. \MR{2186250 (2006i:37005)}

\bibitem{GY}
E.~Glasner and X.~Ye, \emph{Local entropy theory}, Ergodic Theory Dynam.
  Systems \textbf{29} (2009), no.~2, 321--356. \MR{2486773 (2010k:37023)}

\bibitem{Goodman}
T.~N.~T. Goodman, \emph{Relating topological entropy and measure entropy},
  Bull. London Math. Soc. \textbf{3} (1971), 176--180. \MR{0289746 (44 \#6934)}

\bibitem{Goodwyn}
L.~W. Goodwyn, \emph{Topological entropy bounds measure-theoretic entropy},
  Proc. Amer. Math. Soc. \textbf{23} (1969), 679--688. \MR{0247030 (40 \#299)}

\bibitem{Gromov99}
M.~Gromov, \emph{Endomorphisms of symbolic algebraic varieties}, J. Eur. Math.
  Soc. (JEMS) \textbf{1} (1999), no.~2, 109--197. \MR{1694588 (2000f:14003)}

\bibitem{HY}
W.~Huang and X.~Ye, \emph{A local variational relation and applications},
  Israel J. Math. \textbf{151} (2006), 237--279. \MR{2214126 (2006k:37033)}

\bibitem{HYZ1}
W.~Huang, X.~Ye, and G.~H. Zhang, \emph{A local variational principle for
  conditional entropy}, Ergodic Theory Dynam. Systems \textbf{26} (2006),
  no.~1, 219--245. \MR{2201946 (2006j:37015)}

\bibitem{HYZ2}
\bysame, \emph{Relative entropy tuples, relative {U}.{P}.{E}. and {C}.{P}.{E}.
  extensions}, Israel J. Math. \textbf{158} (2007), 249--283. \MR{2342467
  (2008h:37016)}

\bibitem{HYZ}
\bysame, \emph{Local entropy theory for a countable discrete amenable group
  action}, J. Funct. Anal. \textbf{261} (2011), no.~4, 1028--1082.

\bibitem{KLInvention1x}
D.~Kerr and H.~F. Li, \emph{Entropy and the variational principle for actions
  of sofic groups}, Invent. Math., to appear.

\bibitem{KLAJM}
\bysame, \emph{Soficity, amenability, and dynamical entropy}, Amer. J. Math.,
  to appear.

\bibitem{Ko58}
A.~N. Kolmogorov, \emph{A new metric invariant of transient dynamical systems
  and automorphisms in {L}ebesgue spaces}, Dokl. Akad. Nauk SSSR (N.S.)
  \textbf{119} (1958), 861--864. \MR{0103254 (21 \#2035a)}

\bibitem{Lsoficdim}
H.~F. Li, \emph{Sofic mean dimension}, preprint (2011).

\bibitem{Lin-Wei}
E.~Lindenstrauss and B.~Weiss, \emph{Mean topological dimension}, Israel J.
  Math. \textbf{115} (2000), 1--24. \MR{1749670 (2000m:37018)}

\bibitem{Misiu}
M.~Misiurewicz, \emph{A short proof of the variational principle for a
  {$Z^{n}_{+}$}\ action on a compact space}, Bull. Acad. Polon. Sci. S\'er.
  Sci. Math. Astronom. Phys. \textbf{24} (1976), no.~12, 1069--1075.
  \MR{0430213 (55 \#3220)}

\bibitem{MO}
J.~Moulin~Ollagnier, \emph{Ergodic theory and statistical mechanics}, Lecture
  Notes in Mathematics, vol. 1115, Springer-Verlag, Berlin, 1985. \MR{781932
  (86h:28013)}

\bibitem{OW}
D.~S. Ornstein and B.~Weiss, \emph{Entropy and isomorphism theorems for actions
  of amenable groups}, J. Analyse Math. \textbf{48} (1987), 1--141. \MR{910005
  (88j:28014)}

\bibitem{Pestov08}
Vladimir~G. Pestov, \emph{Hyperlinear and sofic groups: a brief guide}, Bull.
  Symbolic Logic \textbf{14} (2008), no.~4, 449--480. \MR{2460675
  (2009k:20103)}

\bibitem{R}
P.~P. Romagnoli, \emph{A local variational principle for the topological
  entropy}, Ergodic Theory Dynam. Systems \textbf{23} (2003), no.~5,
  1601--1610. \MR{2018614 (2004i:37030)}

\bibitem{Vara}
V.~S. Varadarajan, \emph{Groups of automorphisms of {B}orel spaces}, Trans.
  Amer. Math. Soc. \textbf{109} (1963), 191--220. \MR{0159923 (28 \#3139)}

\bibitem{W}
P.~Walters, \emph{An introduction to ergodic theory}, Graduate Texts in
  Mathematics, vol.~79, Springer-Verlag, New York, 1982. \MR{648108
  (84e:28017)}

\bibitem{WZ}
T.~Ward and Q.~Zhang, \emph{The {A}bramov-{R}okhlin entropy addition formula
  for amenable group actions}, Monatsh. Math. \textbf{114} (1992), no.~3-4,
  317--329. \MR{1203977 (93m:28023)}

\bibitem{We}
B.~Weiss, \emph{Actions of amenable groups}, Topics in dynamics and ergodic
  theory, London Math. Soc. Lecture Note Ser., vol. 310, Cambridge Univ. Press,
  Cambridge, 2003, pp.~226--262. \MR{2052281 (2005d:37008)}

\bibitem{Weiss-sofic}
Benjamin Weiss, \emph{Sofic groups and dynamical systems}, Sankhy\=a Ser. A
  \textbf{62} (2000), no.~3, 350--359, Ergodic theory and harmonic analysis
  (Mumbai, 1999). \MR{1803462 (2001j:37022)}

\bibitem{YZ}
X.~Ye and G.~H. Zhang, \emph{Entropy points and applications}, Trans. Amer.
  Math. Soc. \textbf{359} (2007), no.~12, 6167--6186 (electronic). \MR{2336322
  (2008m:37026)}

\end{thebibliography}

\end{document}